\documentclass{amsart}%
\usepackage{amsfonts}
\usepackage{hyperref}
\usepackage{amsmath}
\usepackage{amssymb}
\usepackage{graphicx}%
\newtheorem{theorem}{Theorem}[section]

\newtheorem{definition}[theorem]{Definition}

\newtheorem{lemma}[theorem]{Lemma}

\newtheorem{proposition}[theorem]{Proposition}
\newtheorem{remark}[theorem]{Remark}

\numberwithin{equation}{section}
\begin{document}
\title[CR-Invariants and the Scattering Operator]{CR-Invariants and the Scattering Operator for Complex Manifolds with Boundary}
\author{Peter D. Hislop}
\address{Department of Mathematics \\
University of Kentucky, \\
Lexington, Kentucky 40506--0027 \\
U. S. A.}
\author{Peter A. Perry}
\address{Department of Mathematics \\
University of Kentucky, \\
Lexington, Kentucky 40506--0027 \\
U. S. A.}
\author{Siu-Hung Tang}
\address{Department of Mathematics \\
University of Kentucky \\
Lexington, Kentucky 40506--0027 \\
U. S. A.}
\thanks{Supported in part by NSF\ grant DMS-0503784}
\thanks{Supported in part by NSF\ grant DMS-0408419}
\thanks{Supported in part by NSF\ grants DMS-0408419 and DMS-0503784.}
\maketitle
\tableofcontents

\section{Introduction}

\label{sec.intro}

The purpose of this paper is to describe certain CR-covariant differential
operators on a strictly pseudoconvex CR manifold $M$ as residues of the
scattering operator for the Laplacian on an ambient complex K\"{a}hler
manifold $X$ having $M$ as a `CR-infinity.' We also characterize the CR
$Q$-curvature in terms of the scattering operator. Our results parallel
earlier results of Graham and Zworski \cite{GZ:2003}, who showed that if $X$
is an asymptotically hyperbolic manifold carrying a Poincar\'{e}-Einstein
metric, the $Q$-curvature and certain conformally covariant differential
operators on the `conformal infinity' $M$ of $X$ can be recovered from
the scattering operator on $X$. The results in this paper were
announced in \cite{HPT:2006}.

To describe our results, we first recall some basic notions of CR\ geometry
and recent results \cite{FH:2003}, \cite{GG:2005} concerning CR-covariant
differential operators and CR-analogues of $Q$-curvature. If $M$ is a smooth,
orientable manifold of real dimension $(2n+1$), a \emph{CR-structure} on $M$
is a real hyperplane bundle $H$ on $TM$ together with a smooth bundle map
$J:H\rightarrow H$ with $J^{2}=-1$ that determines an almost complex structure
on $H$. We denote by $T_{1,0}$ the eigenspace of $J$ on $H\otimes\mathbb{C}$
with eigenvalue $+i$; we will always assume that the CR-structure on $M$ is
integrable in the sense that $\left[  T_{1,0},T_{1,0}\right]  \subset T_{1,0}%
$. We will assume that $M$ is orientable, so that the line bundle $H^{\perp
}\subset T^{\ast}M$ admits a nonvanishing global section. A
\emph{pseudo-Hermitian structure} on $M$ is smooth, nonvanishing section
$\theta$ of $H^{\perp}$. The Levi form of $\theta$ is the Hermitian form
$L_{\theta}(v,w)=d\theta(v,Jw)$ on $H$. The CR structure on $M$ is called
\emph{strictly pseudoconvex} if the Levi form is positive definite. Note that
this condition is actually independent of the choice of $\theta$ compatible
with a given orientation of $M$. We will always assume that $M$ is strictly
pseudoconvex in what follows. It follows from strict pseudoconvexity that
$\theta$ is a contact form, and the form $\theta\wedge(d\theta)^{n}$ is a
volume form that defines a natural inner product on $\mathcal{C}^{\infty}(M)$
by integration. The pseudo-Hermitian structure on $M$ also determines a
connection on $TM$, the Tanaka-Webster connection $\nabla_{\theta}$; the basic
data of pseudo-Hermitian geometry are the curvature and torsion of this
connection (see \cite{Tanaka:1975}, \cite{Webster:1978}).

Given a fixed CR-structure $(H,J)$ on $M$, any nonvanishing section
$\overline{\theta}$ of $H^{\perp}$ compatible with a given orientation takes
the form $e^{2\Upsilon}\theta$ for a fixed section $\theta$ of $H^{\perp}$ and
some function $\Upsilon\in\mathcal{C}^{\infty}(M)$. The corresponding Levi
form is given by $L_{\overline{\theta}}=e^{2\Upsilon}L_{\theta}$. In this
sense the CR-structure determines a conformal class of pseudo-Hermitian
structures on $M$.

For strictly pseudoconvex CR-manifolds, Fefferman and Hirachi \cite{FH:2003}
proved the existence of CR-covariant differential operators $P_{k}$ of order
$2k$, $k=1,2,\ldots,n+1$, whose principal parts are $\Delta_{\theta}^{k}$,
where $\Delta_{\theta}$ is the sub-Laplacian on $M$ with respect to the
pseudo-Hermitian structure $\theta$. They exploit Fefferman's construction
(formulated intrinsically by Lee in \cite{Lee:1986}) of a circle bundle
$\mathcal{C}$ over $M$ with a natural conformal structure and a mapping
$\theta\mapsto g_{\theta}$ from pseudo-Hermitian structures on $M$ to Lorentz
metrics on $\mathcal{C}$ that respects conformal classes. They then construct
the conformally covariant differential operators found in
\cite{GJMS:1992} (referred to here as GJMS operators)
on $\mathcal{C}$, and show that these operators pull back to CR-covariant
differential operators on $M$. The CR $Q$-curvature may be similarly defined
as a pullback to $M$ of Branson's $Q$-curvature on the circle bundle
$\mathcal{C}$. Here we will show that the operators $P_{k}$ on $M$ occur as
residues for the scattering operator associated to a natural scattering
problem with $M$ as the boundary at infinity, and that the CR $Q$-curvature
$Q_{\theta}^{CR}$ can be computed from the scattering operator.

To describe the scattering problem, we first discuss its geometric setting.
Recall that if $M$ is an integrable, strictly pseudoconvex CR-manifold of
dimension $(2n+1)$ with $n\geq2$, there is a complex manifold $X$ of complex
dimension $m=n+1$ having $M$ as its boundary so that the CR-structure on $M$
is induced from the complex structure on $X$ (this result is false, in
general, when $n=1$; see \cite{HL:1975}). Let $\varphi$ be a defining function
for $M$ and denote by $\mathring{X}$ the interior of $X$ (we take $\varphi<0$
in $\mathring{X}$). The associated K\"{a}hler metric $g$ on $\mathring{X}$ is
the K\"{a}hler metric with K\"{a}hler form
\begin{equation}
\omega_{\varphi}=-\frac{i}{2}\partial\overline{\partial}\log(-\varphi)
\label{eq.omega}%
\end{equation}
in a neighborhood of $M$, extended smoothly to all of $X$. The metric has the
form
\begin{equation}
\label{eq.metric}g_{\varphi}=-\frac{\eta}{\varphi}+(1-r\varphi)\left(
\frac{d\varphi^{2}}{\varphi^{2}}+\frac{\Theta^{2}}{\varphi^{2}}\right)  .
\end{equation}
in a neighborhood of $M$, where $\eta$ and $\Theta$ have Taylor series to all
orders in $\varphi$ at $\varphi=0$. The boundary values $\Theta|_{M}=\theta$,
and $\eta|_{H}=h$ induce respectively a contact form on $M~$and a Hermitian
metric on $H$. The function $r$ is a smooth function, the transverse
curvature, which depends on the choice of $\varphi$ (see \cite{GL:1988}).
Thus, the conformal class of a Hermitian metric $h$ on $H$, a \emph{subbundle}
of $TM$, is a kind of `Dirichlet datum at infinity' for the metric $g_\varphi$, that
is $(-\varphi) g_\varphi |_{H}=h$.

A motivating example for our work is the case of a strictly pseudoconvex
domain $X\subset\mathbb{C}^{m}$ with Hermitian metric%
\[
g=\sum_{j,k=1}^{m}\frac{\partial^{2}}{\partial z_{j}\partial z_{\overline{k}}%
}\log\left(  -\frac{1}{\varphi}\right)  dz_{j}\otimes dz_{\overline{k}},
\]
where $\varphi$ is a defining function for the boundary of $X$ with
$\varphi<0$ in the interior of $X$. In this example, observe that if
\[
\Theta=\frac{i}{2}\left(  \overline{\partial}\varphi-\partial\varphi\right)
\]
and $\iota:M\rightarrow X$ is the natural inclusion, then $\theta=\iota^{\ast
}\Theta$ is a contact form on $M$ that defines the CR-structure $H=\ker\theta
$. The form $d\theta$ induces the Levi form on $M$ and so defines a
pseudo-Hermitian structure on $M$. Denote by $J$ the almost complex structure
on $H$; the two-form $h=d\theta(~\cdot~,J~\cdot~)$ is a pseudo-Hermitian
metric on $M$. It is not difficult to see that the conformal class of the
pseudo-Hermitian structure on $M$, i.e., its CR-structure, is independent of
the choice of defining function $\varphi$.

It is natural to consider scattering theory for the Laplacian, $\Delta_{g}$,
on $(\mathring{X},g)$, where $X$ is a complex manifold with boundary $M$. As
discussed in what follows, the metric $g$ belongs to the class of $\Theta
$-metrics considered by Epstein, Melrose, and Mendoza \cite{EMM:1991}; see
also the recent paper of Guillarmou and S\'a Barreto \cite{GSb:2006} where
scattering theory for asymptotically complex hyperbolic manifolds (a class
which includes those considered here) is analyzed in depth. Thus, the full
power of the Epstein-Melrose-Mendoza analysis of the resolvent $R(s)=\left(
\Delta_{g}-s(m-s)\right)  ^{-1}$ of $\Delta_{g}$ is available to study
scattering theory on $(\mathring{X},g)$.

For $f\in\mathcal{C}^{\infty}(M)$, $\Re(s)=m/2$, and $s\neq m/2$, there is a
unique solution $u$ of the `Dirichlet problem'%
\begin{align}
\left(  \Delta_{g}-s(m-s)\right)  u  &  =0\label{eq.dp}\\
u  &  =(-\varphi)^{m-s}F+(-\varphi)^{s}G\nonumber\\
\left.  F\right\vert _{M}  &  =f.\nonumber
\end{align}
where $F,G\in\mathcal{C}^{\infty}(X)$. The uniqueness follows from the absence
of $L^{2}$ solutions of the eigenvalue problem for $\Re(s)=m/2$;
this may be proved, for example, using \cite{VW:2004} (see the comments in
\cite{GSb:2006}). Here we will use the explicit formulas for the K\"{a}hler
form and Laplacian obtained in \cite{GL:1988} to obtain the asymptotic
expansions of solutions to the generalized eigenvalue problem.

Unicity for the `Dirichlet problem' (\ref{eq.dp}) implies that the Poisson
map
\begin{align}
\mathcal{P}(s):\mathcal{C}^{\infty}(M) &  \rightarrow\mathcal{C}^{\infty
}(\mathring{X})\label{eq.Poisson.map}\\
f &  \mapsto u\nonumber
\end{align}
and the scattering operator
\begin{align*}
S_{X}(s):\mathcal{C}^{\infty}(M) &  \rightarrow\mathcal{C}^{\infty}(M)\\
f &  \mapsto\left.  G\right\vert _{M}%
\end{align*}
are well-defined. The operator $S_{X}(s)$ depends \textit{a priori} on the
boundary defining function $\varphi$ for $M$. If $\overline{\varphi
}=e^{\upsilon}\varphi$ is another defining function for $M$ and $\upsilon
|_{M}=\Upsilon$, the corresponding scattering operator $\overline{S}_{X}(s)$
is given by
\[
\overline{S}_{X}(s)=e^{-s\Upsilon}S_{X}(s)e^{(s-m)\Upsilon}.
\]
The operator $S_{X}(s)$ admits a meromorphic continuation to the complex
plane, possibly with singularities at $s=0,-1,-2,\cdots$; see
\cite{Melrose:2000} where the scattering operator is described and the problem
of studying its poles and residues is posed, and see \cite{GSb:2006} for a
detailed analysis of the scattering operator. The scattering operator is
self-adjoint for $s$ real. We will show that, with a geometrically natural
choice of the boundary defining function $\varphi$, the residues of
certain poles of $S_{X}(s)$ are CR-covariant differential operators.

To describe the setting for this result, recall that for strictly pseudoconvex
domains $\Omega$ in $\mathbb{C}^{m}$, Fefferman \cite{Fefferman:1976} proved
the existence of a defining function $\varphi$ for $\partial\Omega$ which is
an approximate solution of the complex Monge-Amp\`{e}re equation.

The complex Monge-Amp\`{e}re equation for a function $\varphi\in
\mathcal{C}^{\infty}(\Omega)$ is the equation
\begin{align*}
J\left[  \varphi\right]    & =1\\
\left.  \varphi\right\vert _{\partial\Omega}  & =0
\end{align*}
where $J$ is the complex Monge-Amp\`{e}re operator%
\[
J\left[  \varphi\right]  =\det\left[
\begin{array}
[c]{cc}%
\varphi & \varphi_{j}\\
\varphi_{\overline{k}} & \varphi_{j\overline{k}}%
\end{array}
\right]
\]
We say that $\varphi\in\mathcal{C}^{\infty}(\Omega)$ is an approximate
solution of the complex Monge-Amp\`{e}re equation if
\begin{align*}
J\left[  \varphi\right]   &  =1+\mathcal{O}\left(  \varphi^{m+1}\right)  \\
& \\
\left.  \varphi\right\vert _{\partial\Omega} &  =0
\end{align*}
The K\"{a}hler metric $g$ associated to such an approximate solution $\varphi$
is an approximate K\"{a}hler-Einstein metric on $\Omega$, i.e., $g$ obeys%
\begin{equation}
\operatorname*{Ric}(g)=-(m+1)\omega+\mathcal{O}(\varphi^{m-1}).\label{eq.KE}%
\end{equation}
where $\omega$ is the K\"{a}hler form associated to $\varphi$, and
$\operatorname*{Ric}$ is the Ricci form.

Under certain conditions, Fefferman's result can be `globalized' to the
setting of complex manifolds $X$ with strictly pseudoconvex boundary $M$, as
we discuss below. It follows that $\mathring{X}$ carries an approximate
K\"{a}hler-Einstein metric $g$ in the sense that (\ref{eq.KE}) holds.

We will call a smooth function $\varphi$ defined in a neighborhood of $M$ a
\emph{globally defined approximate solution} of the Monge-Amp\`{e}re equation
on $X$ if for each $p\in M$ there is a neighborhood $U$ of $p$ in $X$ and a
holomorphic coordinate system in $U$ for which $\varphi$ is an approximate
solution of the Monge-Amp\`{e}re equation. As we will show, such a solution
exists if and only if $M$ admits a pseudo-Hermitian structure $\theta$ which
is volume-normalized with respect to some locally defined, closed
$(n+1,0)$-form in a neighborhood of any point $p\in M$ (see section
\ref{sec.monge-ampere.global} where we defined \textquotedblleft
volume-normalized\textquotedblright, and see Burns-Epstein \cite{BE:1990}
where a similar condition is used to construct a global solution of the
Monge-Amp\'ere equation when $\dim M=3$). If $\dim M\geq5$, we can give a more
geometric formulation of this condition. Recall that a CR-manifold is
\emph{pseudo-Einstein} if there is a pseudo-Hermitian structure $\theta$ for
which the Webster Ricci curvature is a multiple of the Levi form (see Lee
\cite{Lee:1988} where this geometric notion is introduced and studied). In
\cite{Lee:1988}, Lee proved that if $\dim M\geq5$, then $M$ admits a
pseudo-Einstein, pseudo-Hermitian structure $\theta$ if and only if $\theta$
is volume-normalized with respect to a closed $(n+1,0)$-form in a neighborhood
of any point $p\in M$. If $\dim M=3$, the pseudo-Einstein condition is vacuous
and must be replaced by a more stringent condition; see section
\ref{sec.monge-ampere.global} in what follows. If $X$ is a pseudoconvex domain
in $\mathbb{C}^{m}$, this condition is trivially satisfied since the
pseudo-Hermitian structure induced by the Fefferman approximate solution is
volume-normalized with respect to the restriction of $\zeta=dz^{1}\wedge
\cdots\wedge dz^{m}$ to $M$.

\begin{theorem}
\label{thm.1}Let $X$ be a complex manifold of complex dimension $m=n+1$ with
strictly pseudoconvex boundary $M$. Let $g$ be the K\"{a}hler metric on $X$
associated to the K\"{a}hler form (\ref{eq.omega}), and let $S_{X}(s)$ be the
scattering operator for $\Delta_{\varphi}$. Finally, suppose that
$\Delta_{\varphi}$ has no $L^{2}$-eigenvalues. Then $S_{X}(s)$ has simple
poles at the points $s=m/2+k/2$, $k\in\mathbb{N}$, and%
\begin{equation}
\operatorname*{Res}_{s=m/2+k/2}S_{X}(s)=c_{k}P_{k}, \label{eq.ResSk}%
\end{equation}
where the $P_{k}$ are differential operators of order $2k$, and
\begin{equation}
c_{k}=\frac{(-1)^{k}}{2^{k}k!(k-1)!}. \label{eq.ck}%
\end{equation}
If $g$ is an approximate K\"{a}hler-Einstein metric given by a globally
defined approximate solution of the Monge-Amp\`{e}re equation, then for $1\leq
k\leq m$, the operators $P_{k}$ are CR-covariant differential operators.
\end{theorem}

\begin{remark}
It is not difficult to show that, for generic compactly supported
perturbations of the metric, $L^{2}$-eigenvalues are absent. Our analysis
applies if only the metric $g$ has the form (\ref{eq.metric}) in a
neighborhood of $M$.
\end{remark}

\begin{remark}
We view the operators $P_{k}$ as operators on $\mathcal{C}^{\infty}(M)$; if
one instead views these operators as acting on appropriate density bundles
over $M$ they are actually invariant operators. Gover and Graham
\cite{GG:2005} showed that the CR-covariant differential operators $P_{k}$ are
logarithmic obstructions to the solution of the Dirichlet problem
(\ref{eq.dp}) when $X$ is a pseudoconvex domain in $\mathbb{C}^{m}$ with a
metric of Bergman type, but did not identify them as residues of the
scattering operator.
\end{remark}

It follows from the self-adjointness ($s$ real) and conformal covariance of
$S_{X}(s)$ that the operators $P_{k}$ are self-adjoint and conformally
covariant. As in \cite{GZ:2003}, the analysis centers on the Poisson map
$\mathcal{P}(s)$ defined in (\ref{eq.Poisson.map}). As shown in
\cite{EMM:1991}, the Poisson map is analytic in $s$ for $\operatorname{Re}%
(s)>m/2$. Moreover, at the points $s=m/2+k/2$, $k=1,2,\cdots$, the Poisson
operator takes the form
\[
\mathcal{P}(s)f=(-\varphi)^{m/2-k/2}F+[(-\varphi)^{m/2+k/2}\log(-\varphi)]G
\]
for functions $F,G\in\mathcal{C}^{\infty}(X)$ with
\[
F|_{M}=f,~G|_{M}=c_{k}P_{k}f.
\]
Here $P_{k}$ are differential operators determined by a formal power series
expansion of the Laplacian (see Lemma \ref{lemma.formal}), and are the same
operators that appear as residues of the scattering operator at points
$s=m/2+k/2$. An important ingredient in the analysis is the asymptotic form of
the Laplacian due to Lee and Melrose \cite{LM:1982} and refined by Graham and
Lee in \cite{GL:1988}.

If the defining function $\varphi$ is an approximate solution of the complex
Monge-Amp\`{e}re equation, the differential operators $P_{k}$, $1\leq k\leq
m\,$, can be identified with the GJMS\ operators owing to the characterization
of $\mathcal{P}(s)f$ described above (see Proposition 5.4 in \cite{GG:2005};
the argument given there for pseudoconvex domains easily generalizes to the
present setting).

Explicit computation shows that, for an approximate K\"{a}hler-Einstein metric
$g$, the first operator has the form
\[
P_{1}=c_{1}(\Delta_{b}+n(2(n+1))^{-1}R),
\]
where $\Delta_{b}$ is the sub-Laplacian on $X$ and $R$ is the Webster scalar
curvature, i.e., $P_{1}$ is the CR-Yamabe operator of Jerison and Lee
\cite{JL:1984}.

The CR $Q$-curvature is a pseudo-Hermitian invariant realized as the pullback
to $M$ of the $Q$-curvature of the circle bundle $\mathcal{C}$.

\begin{theorem}
\label{thm.2}Suppose that $X$ is a complex manifold with strictly pseudoconvex
boundary $M$, and suppose that $g$ is an approximate K\"{a}hler-Einstein
metric given by a globally defined approximate solution of the
Monge-Amp\`{e}re equation. Let $S_{X}(s)$ be the associated scattering
operator. The formula
\[
c_{m}Q_{\theta}^{CR}=\lim_{s\rightarrow m}S_{X}(s)1
\]
holds, where $c_{m}$ is given by (\ref{eq.ck}).
\end{theorem}

It follows from Theorem \ref{thm.1} and the conformal covariance of $S_{X}(s)$
that if $\overline{\theta}=e^{2\Upsilon}\theta$, then
\[
e^{2m\Upsilon}Q_{\overline{\theta}}^{CR}=Q_{\theta}^{CR}+P_{m}\Upsilon
\]
as was already shown in Fefferman-Hirachi \cite{FH:2003}. From this it follows
that the integral $\int_{M}Q_{\theta}^{CR}~\psi$ is a CR-invariant (recall
that $\psi$ is the natural volume form on $M$ defined by the contact form
$\theta$). We remark that the integral of $Q_{\theta}^{CR}$ vanishes for any
three-dimensional CR-manifold because the integrand is a total divergence (see
\cite{FH:2003}, Proposition 3.2 and comments below), while under the condition
of our Theorem \ref{thm.2}, there is a pseudo-Hermitian structure for which
$Q_{\theta}^{CR}=0$ (see \cite{FH:2003}, Proposition 3.1). In our case, if
$\varphi$ is a globally defined approximate solution of the Monge-Amp\`{e}re
equation, the induced contact form $\theta=\left(  i/2\right)  (\overline
{\partial}\varphi-\partial\varphi)$ on $M$ is an `invariant contact form' in
the language of \cite{FH:2003}, and they show in Proposition 3.1 that
$Q_{\theta}^{CR}=0$ for an invariant contact form. Thus it is not clear at
present under what circumstances this invariant is nontrivial for a general,
strictly pseudoconvex manifold.

Finally, we prove a CR-analogue of Graham and Zworski's result (\cite{GZ:2003}%
, Theorem 3) using scattering theory.

\begin{theorem}
\label{thm.3}Suppose that $X$ is a compact complex manifold with strictly
pseudoconvex boundary $M$, and $g$ is an approximate-K\"{a}hler-Einstein
metric given by a globally defined approximate solution of the
Monge-Amp\`{e}re equation. Then%
\begin{equation}
\operatorname*{vol}\nolimits_{g}\left\{  -\varphi>\varepsilon\right\}
=c_{0}\varepsilon^{-n-1}+c_{1}\varepsilon^{-n}+\cdots+c_{n}\varepsilon
^{-1}+L\log(-\varepsilon)+V+o(1). \label{eq.vol}%
\end{equation}
where%
\[
L=c_{m}\int_{M}Q_{\theta}^{CR}~\psi=0
\]

\end{theorem}

We remark that Seshadri \cite{Seshadri:2004} already showed that $L$ is, up to
a constant, the integral of $Q_{\theta}^{CR}$. It is worth noting that our
choice of defining function differs from Seshadri's.

\textbf{Acknowledgements}. We are grateful to Charles Epstein, Robin Graham,
John Lee, and Ant\^{o}nio S\'{a} Barreto for helpful discussions and
correspondence. We specifically acknowledge Robin Graham for a private lecture
on the Monge-Amp\`{e}re equation on which section \ref{sec.monge-ampere} is
based. PAP thanks the National Center for Theoretical Science at the National
Tsing Hua University, Hsinchu, Taiwan, for hospitality during part of the time
that this work was done. SHT thanks the Department of Mathematics at the
University of Kentucky for hospitality during part of the time that this work
was done.

\section{Geometric Preliminaries}

\label{sec.prelim}

\subsection{CR\ Manifolds}

\label{sec.prelim.CR}

Suppose that $M$ is a smooth orientable manifold of real dimension $2n+1$, and
let $\mathbb{C}TM=TM\otimes_{\mathbb{R}}\mathbb{C}$ be the complexified
tangent bundle on $M$. A \emph{CR-structure }on $M$ is a complex
$n$-dimensional subbundle $\mathcal{H}$ of $\mathbb{C}TM$ with the property
that $\mathcal{H}\cap\overline{\mathcal{H}}=\left\{  0\right\}  $. If, also,
$\left[  \mathcal{H},\mathcal{H}\right]  \subseteq\mathcal{H}$, we say that
the CR-structure is \emph{integrable}. If we set $H=\operatorname{Re}%
\mathcal{H}$, then the bundle $H$ has real codimension one in $TM$. The map
\begin{align*}
J :H  &  \rightarrow H\\
V+\overline{V}  &  \mapsto i(V-\overline{V})
\end{align*}
satisfies $J^{2}=-I$ and gives $H$ a natural complex structure.

Since $M$ is orientable, there is a nonvanishing one-form $\theta$ on $M$ with
$\ker\theta=H$. This form is unique up to multiplication by a positive, nonvanishing function
$f \in \mathcal{C}^{\infty} (M)$. A choice of such a one-form $\theta$ is
called a \emph{pseudo-Hermitian structure} on $M$. The \emph{Levi form} is
given by%
\begin{equation}
L_{\theta}(V,\overline{W})=-id\theta(V,\overline{W}). \label{eq.Levi}%
\end{equation}
for $V,W\in\mathcal{H}$ (here $d\theta$ is extended to $\mathcal{H}$ by
complex linearity). Note that
\begin{equation}
L_{f\theta}=fL_{\theta} \label{eq.Levi.conformal}%
\end{equation}
since $\theta$ annihilates $\mathcal{H}$. If $d\theta$ is nondegenerate, then
there is a unique real vector field $T$ on $M$, the \emph{characteristic
vector field} $T$, with the properties that $\theta(T)=1$ and $T~\lrcorner
~d\theta=0$. If $\left\{  W_{\alpha}\right\}  $ is a local frame for
$\mathcal{H}$ (here $\alpha$ ranges from $1$ to $n$), then the vector fields
$\left\{  W_{\alpha},W_{\overline{\alpha}},T\right\}  $ form a local frame for
$\mathbb{C}TM$. If we choose $(1,0)$-forms $\theta^{\alpha}$ dual to the
$W_{\alpha}$ then $\left\{  \theta^{\alpha},\theta^{\overline{\alpha}}%
,\theta\right\}  $ forms a dual coframe for $\mathbb{C}TM$. We say that
$\{\theta^{\alpha}\}$ forms an admissible coframe dual to $\{W^{\alpha}\}$ if
$\theta^{\alpha}(T)=0$ for all $\alpha$. The integrability condition is
equivalent to the condition that%
\begin{equation}
d\theta=d\theta^{\alpha}=0~\operatorname{mod}~\left\{  \theta,\theta^{\alpha
}\right\}  \label{eq.CR.integrable}%
\end{equation}
The Levi form is then given by
\begin{equation}
L_{\theta}=h_{\alpha\overline{\beta}}\theta^{\alpha}\wedge\theta
^{\overline{\beta}} \label{eq.Levi.loc}%
\end{equation}
for a Hermitian matrix-valued function $h_{\alpha\overline{\beta}}$. We will
use $h_{\alpha\overline{\beta}}$ to raise and lower indices in this article.

We will say that a given CR-structure is \emph{strictly pseudoconvex} if
$L_{\theta}$ is positive definite. Note that (up to sign)\ this condition is
independent of the choice of pseudo-Hermitian structure $\theta$.

In what follows, we will always suppose that $M$ is orientable and that $M$
carries a strictly pseudoconvex, integrable CR-structure. In this case, the
pseudo-Hermitian geometry of $M$ can be understood in terms of the
\emph{Tanaka-Webster connection} on $M$ (see Tanaka \cite{Tanaka:1975} and
Webster \cite{Webster:1978}). With respect to the frame discussed above, the
Tanaka-Webster connection is given by
\begin{equation}
\nabla W_{\alpha}=\omega_{\alpha}^{~~\beta}\otimes W_{\beta},~\nabla T=0
\label{eq.webster}%
\end{equation}
for connection one-forms $\omega_{\alpha}^{~~\beta}$ obeying the structure
equations%
\begin{align*}
d\theta^{\alpha}  &  =\theta^{\beta}\wedge\omega_{\alpha}^{~~\beta}%
+\theta\wedge\tau^{\alpha}\\
d\theta &  = ih_{\alpha\overline{\beta}}\theta^{\alpha}\wedge\theta
^{\overline{\beta}}%
\end{align*}
where the torsion one-forms are given by%
\[
\tau^{\alpha}=A_{~~\overline{\beta}}^{\alpha}\theta^{\overline{\beta}},
\]
with $A_{\alpha\beta}=A_{\beta\alpha}.$ The connection obeys the compatibility
condition%
\[
dh_{\alpha\overline{\beta}}= \omega_{\alpha\overline{\beta}}+\omega
_{\overline{\beta}\alpha}.
\]
with the Levi form described in (\ref{eq.Levi}) and (\ref{eq.Levi.loc}).

\subsection{Complex Manifolds with CR Boundary}

\label{sec.prelim.complex}

Now suppose that $X$ is a compact complex manifold of dimension $m=n+1$ with
boundary $\partial X=M$. We will denote by $\mathring{X}$ the interior of $X$.
The manifold $M$ inherits a natural CR-structure from the complex structure of
the ambient manifold. We will suppose that that $M$ is strictly pseudoconvex;
such a structure, induced by the complex structure of the ambient manifold, is
always integrable.

We will suppose that $\varphi\in\mathcal{C}^{\infty}(X)$ is a defining
function for $M$, i.e., $\varphi<0$ in $\mathring{X}$, $\varphi=0$ on $M$, and
$d\varphi(p)\neq0$ for all $p\in M$. We will further suppose that $\varphi$
has no critical points in a collar neighborhood of $M$ so that the level sets
$M^{\varepsilon}=\varphi^{-1}(-\varepsilon)$ are smooth manifolds for all
$\varepsilon$ sufficiently small.

Associated to the defining function $\varphi$ is the K\"{a}hler form%
\begin{equation}
\omega_{\varphi} =-\frac{i}{2}\partial\overline{\partial}\log(-\varphi)
=\frac{i}{2}\left(  \frac{\partial\overline{\partial}\varphi}{-\varphi}+
\frac{\partial\varphi\wedge\overline{\partial}\varphi}{\varphi^{2}}\right)
\label{eq.Kahler}%
\end{equation}
We will study scattering on $X$ with the metric induced by the K\"{a}hler form
(\ref{eq.Kahler}). Since we can cover a neighborhood of $M$ in $X$ by
coordinate charts, it suffices to consider the situation where $U$ is an open
subset of $\mathbb{C}^{m}$ and $\varphi:U\rightarrow\mathbb{R}$ is a smooth
function with no critical points in $U$, the set $\left\{  \varphi<0\right\}
$ is biholomorphically equivalent to a boundary neighborhood in $X$, and
$\{\varphi=0\}$ is diffeomorphic to the corresponding boundary neighborhood in
$M$. We will now describe the asymptotic geometry near $M$, recalling the
ambient metric of \cite{GL:1988} and computing the asymptotic form of the
metric and volume form.

The manifolds $M^{\varepsilon}$ inherit a natural CR-structure from the
ambient manifold $X$ with
\[
\mathcal{H}^{\varepsilon}=\mathbb{C}TM^{\varepsilon}\cap T^{1,0}U.
\]
Given a defining function $\varphi$, we define a one-form%
\[
\Theta=\frac{i}{2}\left(  \overline{\partial}-\partial\right)  \varphi
\]
and let
\[
\theta_{\varepsilon}=\iota_{\varepsilon}^{\ast}\Theta
\]
where $\iota_{\varepsilon}:M^{\varepsilon}\rightarrow U$ is the natural
embedding. The contact form
$\theta_{\varepsilon}$ gives $M^{\varepsilon}$ a pseudo-Hermitian structure.
We will denote by $\mathcal{H}$ the subbundle of $T^{1,0}U$
whose fibre over $M^{\varepsilon}$ is $\mathcal{H}^\epsilon$. Note that%
\[
d\Theta=i\partial\overline{\partial}\varphi.
\]
and the Levi form on $M^{\varepsilon}$ is given by
\[
L_{\theta_{\varepsilon}}=-id\theta_{\varepsilon}%
\]
We will assume that each $M^{\varepsilon}$ is strictly pseudoconvex, i.e.,
$L_{\theta_{\varepsilon}}$ is positive definite for all sufficiently small
$\varepsilon>0$. To simplify notation, we will write $\theta$
for $\theta_\epsilon$, suppressing the $\epsilon$, as the meaning will be clear
from the context.

\subsubsection{Ambient Connection}

In order to describe the asymptotic geometry of $X$, we recall the ambient
connection defined by Graham and\ Lee \cite{GL:1988} that extends the
Tanaka-Webster connection on each $M^{\varepsilon}$ to $\mathbb{C}TU$. First
we recall the following lemma from \cite{LM:1982}.

\begin{lemma}
\label{lemma.xi} There exists a unique $(1,0)$-vector field $\xi$ on $U$ so
that:%
\begin{equation}
\partial\varphi(\xi)=1 \label{eq.xi.1}%
\end{equation}
and%
\begin{equation}
\xi~\lrcorner~\partial\overline{\partial}\varphi=r~\overline{\partial}%
\varphi\label{eq.xi.2}%
\end{equation}
for some $r\in\mathcal{C}^{\infty}(U)$.
\end{lemma}

The smooth function $r$ in (\ref{eq.xi.2}) is called the \emph{transverse
curvature}.

We decompose $\xi$ into real and imaginary parts,
\begin{equation}
\xi=\frac{1}{2}(N-iT), \label{eq.xi.NT}%
\end{equation}
where $N$ and $T$ are real vector fields on $U$. It easily follows from
(\ref{eq.xi.NT}) that%
\[
d\varphi(N)=2,~\theta(N)=0
\]
and%
\[
\theta(T)=1,~T~\lrcorner~d\theta=0.
\]
Thus $T$ is the characteristic vector field for each $M^{\varepsilon}$, and
$N$ is normal to each $M^{\varepsilon}$.

Let $\left\{  W_{\alpha}\right\}  $ be a frame for $\mathcal{H}$. It follows
from Lemma \ref{lemma.xi} that $\left\{  W_{\alpha},W_{\overline{\alpha}%
},T\right\}  $ forms a local frame for $\mathbb{C}TM^{\varepsilon}$, while
$\left\{  W_{\alpha},W_{\overline{\alpha}},\xi,\overline{\xi}\right\}  $ forms
a local frame for $\mathbb{C}TU$. If $\left\{  \theta^{\alpha}\right\}  $ is a
dual coframe for $\left\{  W_{\alpha}\right\}  $, then $\left\{
\theta^{\alpha},\theta^{\overline{\alpha}},\theta\right\}  $ is a dual coframe
for $\mathbb{C}TM^{\varepsilon}$, while $\left\{  \theta^{\alpha}%
,\theta^{\overline{\alpha}},\partial\varphi,\overline{\partial}\varphi
\right\}  $ is a dual coframe for $\mathbb{C}TU$. The Levi form on each
$\mathcal{H}^{\varepsilon}$ is given by%
\[
L_{\theta}=h_{\alpha\overline{\beta}}~\theta^{\alpha}\wedge\theta
^{\overline{\beta}}%
\]
for a Hermitian matrix-valued function $h_{\alpha\overline{\beta}}$. We will
use $h_{\alpha\overline{\beta}}$ to raise and lower indices. We will set%
\[
W_{m}=\xi,~~W_{\overline{m}}=\overline{\xi},~~,\theta^{m}=\partial
\varphi,~\theta^{\overline{m}}=\overline{\partial}\varphi.
\]
In what follows, repeated Greek indices are summed from $1$ to $n$ and
repeated Latin indices are summed from $1$ to $m=n+1$.

The following important lemma decomposes the form $d\Theta$ into `tangential'
and `transverse' components.

\begin{lemma}
\label{lemma.hr}The formula%
\[
\partial\overline{\partial}\varphi=h_{\alpha\overline{\beta}}~\theta^{\alpha
}\wedge\theta^{\overline{\beta}}+r~\partial\varphi\wedge\overline{\partial
}\varphi
\]
holds.
\end{lemma}

Graham and Lee \cite{GL:1988} proved:

\begin{proposition}
\label{prop.ambient.connection}There exists a unique linear connection
$\nabla$ on $U$ so that

\begin{description}
\item[(a)] For any vector fields $X$ and $Y$ on $U$ tangent to some
$M^{\varepsilon}$, $\nabla_{X}Y=\nabla_{X}^{\varepsilon}Y$ where
$\nabla^{\varepsilon}$ is the pseudo-Hermitian connection on $M^{\varepsilon}$.

\item[(b)] $\nabla$ preserves $\mathcal{H}$, $N$, $T$, and $L_{\theta}$; that
is, $\nabla_{X}\mathcal{H}\subset\mathcal{H}$ for any $X\in\mathbb{C}TU$, and
$\nabla T=\nabla N=\nabla L_{\theta}=0.$

\item[(c)] If $\left\{  W_{\alpha}\right\}  $ is a frame for $\mathcal{H}$,
and $\left\{  \theta^{\alpha},\partial\varphi\right\}  $ is the dual
$(1,0)$-coframe on $U$, then%
\begin{equation}
d\theta^{\alpha}=\theta^{\beta}\wedge\varphi_{\beta}^{~~\alpha}-i\partial
\varphi\wedge\tau^{\alpha}+i(W^{\alpha}r)d\varphi\wedge\theta+\frac{1}%
{2}r~d\varphi\wedge\theta^{\alpha} \label{eq.ambient}%
\end{equation}

\end{description}
\end{proposition}

The connection $\nabla$ is called the \emph{ambient connection}.

\subsubsection{K\"{a}hler Metric}

Using Lemma \ref{lemma.hr}, we can also compute the K\"{a}hler form%
\begin{equation}
\omega_{\varphi}=\frac{i}{2}\left(  \frac{1}{-\varphi}h_{\alpha\overline
{\beta}}\theta^{\alpha}\wedge\theta^{\overline{\beta}}+\frac{1-r\varphi
}{\varphi^{2}}\partial\varphi\wedge\overline{\partial}\varphi\right)  .
\label{eq.Kahler.2}%
\end{equation}
The induced Hermitian metric is%
\begin{equation}
g_{\varphi}=\frac{1}{-\varphi}h_{\alpha\overline{\beta}}\theta^{\alpha}%
\otimes\theta^{\overline{\beta}}+\frac{1-r\varphi}{\varphi^{2}}\partial
\varphi\otimes\overline{\partial}\varphi. \label{eq.gphi}%
\end{equation}
It is easily computed that
\[
g_{\varphi}(N,N)=4\frac{1-r\varphi}{\varphi^{2}}%
\]
so that the outward unit normal field associated to the surfaces
$M^{\varepsilon}$ is%
\begin{equation}
\nu=\frac{-\varphi}{2\sqrt{1-r\varphi}}N \label{eq.nu}%
\end{equation}
We note for later use that the induced volume form $\omega_{\varphi}^{m}$ is
given by%
\begin{equation}
\omega_{\varphi}^{m}=\left(  \frac{i}{2}\right)  ^{m}\left(  \frac{1-r\varphi
}{(-\varphi)^{m+1}}\det\left(  h_{\alpha\overline{\beta}}\right)  ~\theta
^{1}\wedge\theta^{\overline{1}}\wedge\cdots\wedge\theta^{m}\wedge
\theta^{\overline{m}}\right)  \label{eq.volume}%
\end{equation}
while%
\begin{equation}
\left.  \nu~\lrcorner~\omega_{\varphi}^{m}\right\vert _{M^{\varepsilon}}%
=\frac{m}{2^{n-1}}\frac{1-r\varepsilon}{\varepsilon^{m}}\left(  d\theta
_{\varepsilon}\right)  ^{n}\wedge\theta_{\varepsilon} \label{eq.surface}%
\end{equation}
We will set%
\begin{equation}
\psi=\frac{m}{2^{n-1}}\left(  d\theta\right)  ^{n}\wedge\theta\label{eq.psi}%
\end{equation}

We also note for later use that if $u\in\mathcal{C}^{\infty}(X)$ and%
\begin{equation}
du=u_{\alpha}\theta^{\alpha}+u_{\overline{\alpha}}\theta^{\overline{\alpha}%
}+u_{m}\partial\varphi+u_{\overline{m}}\overline{\partial}\varphi\label{eq.du}%
\end{equation}
then%
\begin{equation}
\left\vert du\right\vert _{g_{\varphi}}^{2}=-\varphi h^{\alpha\overline{\beta
}}u_{\alpha}u_{\overline{\beta}}+\frac{\varphi^{2}}{1-r\varphi}u_{m}%
u_{\overline{m}} \label{eq.du.norm}%
\end{equation}


\subsection{The Laplacian on $X$}

\label{sec.prelim.Laplace}

The Laplacian on the K\"{a}hler manifold $(X,\omega_{\varphi})$ is the
operator\footnote{Note that our definition differs from that of Graham and Lee
by an overall factor of -$1/4$.}%
\begin{align}
\Delta_{\varphi}u  &  =\operatorname*{Tr}\left(  i\partial\overline{\partial
}u\right) \label{eq.Bergmann.Laplacian}\\
&  =g^{j\overline{k}}u_{j\overline{k}}\nonumber
\end{align}
for $u\in\mathcal{C}^{\infty}(X)$, where we now write $\Delta_{\varphi}$
rather than $\Delta_{g}$ to emphasize the dependence of $\Delta$ on the
boundary defining function $\varphi$.

Graham and Lee \cite{GL:1988} computed the Laplacian in a collar neighborhood
of $M$, separating `normal' and `tangential' parts. To state their results,
recall that the sub-Laplacian is defined on each $M^{\varepsilon}$ by%
\begin{equation}
\Delta_{b}u=\left(  u_{\alpha}^{~~\alpha}+u_{\overline{\beta}}^{~~\overline
{\beta}}\right)  \label{eq.sublap}%
\end{equation}
where covariant derivatives are taken with respect to the Tanaka-Webster
connection on $M^{\varepsilon}$.

Graham and Lee \cite{GL:1988} proved:

\begin{theorem}
The formula%
\begin{equation}
\Delta_{\varphi}=\frac{\varphi}{4}\left[  \dfrac{-\varphi}{1-r\varphi}\left(
N^{2}+T^{2}+2rN+2X_{r}\right)  -2\Delta_{b}+2nN\right]  \label{eq.lap}%
\end{equation}
holds, where%
\[
X_{r}=r^{\alpha}W_{\alpha}+r^{\overline{\alpha}}W_{\overline{\alpha}}.
\]

\end{theorem}

It will be useful to recast the above formula for $\Delta_{\varphi}u$ in terms
of $x=-\varphi$. Note that
\begin{equation}
N=2\frac{\partial}{\partial\varphi}=-2\frac{\partial}{\partial x}
\label{eq.N.x}%
\end{equation}
so that%
\begin{align*}
\Delta_{\varphi}u  &  =\left(  \frac{1}{1+rx}\right)  \left(  x\frac{\partial
}{\partial x}\right)  ^{2}u-(n+1)x\frac{\partial}{\partial x}\\
&  +\frac{1}{4}\left(  \frac{x^{2}}{1+rx}\right)  \left(  T^{2}u-2ru_{x}%
+2X_{r}u\right) \\
&  +\frac{1}{4}x\left(  -2\Delta_{b}u\right)
\end{align*}
We think of $\Delta_{\varphi}$ as a variable-coefficient differential operator
with respect to vector fields $\left(  x\partial_{x}\right)  $ and vector
fields tangent to the boundary $M$. In a neighborhood of $M$ we have%
\begin{equation}
\Delta_{\varphi}\sim\sum_{k\geq0}x^{k}L_{k} \label{eq.Deltag.exp}%
\end{equation}
for differential operators $L_{k}$, where the indicial operator $L_0$ is%
\begin{equation}
L_{0}=-\left(  \left(  x\frac{\partial}{\partial x}\right)  ^{2}%
-mx\frac{\partial}{\partial x}\right)  \label{eq.normal}%
\end{equation}
and the operator $L_1$ is
\begin{equation}
L_{1}=\frac{1}{4}\left(  -2\Delta_{b}u-4r_{0}x\frac{\partial}{\partial
x}-4r_{0}\left[  \left(  x\frac{\partial}{\partial x}\right)  ^{2}%
-x\frac{\partial}{\partial x}\right]  \right)  \label{eq.next}%
\end{equation}
where%
\[
r=r_{0}+\mathcal{O}(x) .
\]


\subsection{The Complex Monge-Amp\`{e}re Equation}

\label{sec.monge-ampere}

\subsubsection{Local Theory}

\label{sec.monge-ampere.local}

Let $\Omega$ be a domain in $\mathbb{C}^{m}$ with smooth boundary. The complex
Monge-Amp\`{e}re equation is the nonlinear equation%
\begin{align*}
J\left[  u \right]   &  =1\\
\left.  u\right\vert _{\partial\Omega}  &  =0
\end{align*}
for a function $u\in\mathcal{C}^{\infty}(\Omega)$, $u>0$ on $\Omega$, where $J\left[  u
\right]  $ is the Monge-Amp\`{e}re operator:%
\begin{equation}
J\left[  u \right]  =(-1)^{m}\det\left(
\begin{array}
[c]{cc}%
u & u_{\overline{j}}\\
u_{i} & u_{i\overline{j}}%
\end{array}
\right) . \label{eq.J}%
\end{equation}
If $u$ solves the complex Monge-Amp\`{e}re equation then
\[
- \left(  \log\left(  \frac{1}{u}\right)  \right)  _{j\overline{k}}%
~dz^{j}\otimes~dz^{\overline{k}}%
\]
is a K\"{a}hler-Einstein metric.

Fefferman \cite{Fefferman:1976} showed that there is a smooth function
$\psi\in\mathcal{C}^{\infty}(\Omega)$ that satisfies
\begin{align*}
J[\varphi]  &  =1+\mathcal{O}(\varphi^{m+1})\\
& \\
\left.  \varphi\right\vert _{\partial\Omega}  &  =0
\end{align*}
and that $\psi$ is uniquely determined up to order $m+1$. Cheng and Yau
\cite{CY:1980} showed the existence of an exact solution belonging to
$\mathcal{C}^{\infty}(\Omega)\cap C^{m+3/2-\epsilon}(\overline{\Omega})$, while
Lee and Melrose \cite{LM:1982} showed that the exact solution has an
asymptotic expansion with logarithmic terms beginning at order $m+2$.

We will show that Fefferman's local approximate solution of the
Monge-Amp\`{e}re equation \cite{Fefferman:1976} can be globalized to an
approximate solution of the Monge-Amp\`{e}re equation near the boundary of a
complex manifold $X$. We will see later that, to globalize Fefferman's
construction, we need to impose a geometric condition on the CR-structure of
$M$ inherited from the complex structure of $X$. For the convenience of the
reader, we review the properties of the operator $J$ under a holomorphic
coordinate change and the connection between solutions of the Monge-Amp\'ere
equation and K\"ahler-Einstein metrics.

If $f:\Omega\subset\mathbb{C}^{m}\rightarrow\mathbb{C}^{m}$ is holomorphic,
then $f^{\prime}$ denotes the matrix%
\[
(f^{\prime})_{jk}=\frac{\partial f_{j}}{\partial z_{k}} .%
\]

\begin{lemma}
\label{lemma.f}Let $f$ be a local biholomorphism. Then, for any local, smooth
function $u$ on $\Omega$,%
\[
J\left[  \left\vert \det(f^{\prime})\right\vert ^{-2/(m+1)}\left(  u\circ
f\right)  \right]  =J\left[  u \right]  \circ f .
\]

\end{lemma}

A proof was given by Fefferman in \cite{Fefferman:1976}. Here we give an
alternative proof using the following identity.

\begin{lemma}
\label{lemma.J.alt}The formula%
\begin{equation}
J\left[  u \right]  =u^{m+1}\det\left[  \left(  \log\left(  \frac{1}%
{u}\right)  \right)  _{j\overline{k}}\right]  \label{eq.J.alt}%
\end{equation}
holds.
\end{lemma}

\begin{proof}
Using row-column operations, one proves that%
\begin{equation}
\det\left(
\begin{array}
[c]{cc}%
u & u_{\overline{k}}\\
u_{j} & u_{j\overline{k}}%
\end{array}
\right)  =u\det\left(  u_{j\overline{k}}-\frac{u_{j}u_{\overline{k}}}%
{u}\right)  . \label{eq.J1}%
\end{equation}
On the other hand, the identity%
\[
\left(  \log\left(  \frac{1}{u}\right)  \right)  _{j\overline{k}}%
=-\frac{u_{j\overline{k}}}{u}+\frac{u_{j}u_{\overline{k}}}{u^{2}}%
\]
shows that%
\begin{equation}
J[u] = (-1)^{m}u\det\left(  u_{j\overline{k}}-\frac{u_{j}u_{\overline{k}}}{u}\right)
   = u^{m+1}\det\left(  \left(  \log\left(  \frac{1}{u}\right)  \right)
_{j\overline{k}}\right)  . \label{eq.J2}%
\end{equation}
Combining (\ref{eq.J1}) and (\ref{eq.J2}) shows that (\ref{eq.J.alt}) holds.
\end{proof}

We can use the formula (\ref{eq.J.alt}) to show that if $u$ solves the
Monge-Amp\`{e}re equation, then $u$ is the K\"{a}hler potential of a
K\"{a}hler-Einstein metric. Recall that if
\[
g=v_{j\overline{k}}~dz^{j}\otimes~dz^{\overline{k}}%
\]
then the Ricci curvature is
\[
R_{a\overline{b}}=-\left[  \log\det(v_{j\overline{k}})\right]  _{a\overline
{b}}%
\]
Now let
\[
v=\log\left(  \frac{1}{u}\right)
\]
where $J\left[  u \right]  =1$. Then%
\begin{align*}
R_{a\overline{b}}  &  =-\left[  \log\det(v_{j\overline{k}})\right]
_{a\overline{b}}\\
&  =-\left[  \log\left(  u^{-(m+1)}\right)  \right]  _{a\overline{b}}\\
&  =-(m+1)\left(  \log\left(  \frac{1}{u}\right)  \right)  _{a\overline{b}}\\
&  =-(m+1)g_{a\overline{b}}%
\end{align*}
which is the Einstein equation.

Now we prove Lemma \ref{lemma.f}. First, we compute%
\begin{align}
\left[  \log\left(  \left\vert \det f^{\prime}\right\vert ^{-2/(m+1)}~u\circ
f\right)  \right]  _{j\overline{k}} &  =\frac{-1}{m+1}\left[  \log\left(
\left\vert \det f^{\prime}\right\vert ^{2}\right)  \right]  _{j\overline{k}%
}\label{eq.med}\\
&  +\left[  \log\left(  u\circ f\right)  \right]  _{j\overline{k}}\nonumber\\
&  =\left[  \log\left(  u\circ f\right)  \right]  _{j\overline{k}}\nonumber
\end{align}
where the first right-hand term vanishes because $\left\vert \det f^{\prime
}\right\vert ^{2}=\left(  \det f^{\prime}\right)  \overline{\left(  \det
f^{\prime}\right)  }$ and $\det f^{\prime}$ is holomorphic. We note that the
vanishing of the first term also shows that the K\"{a}hler metric with
K\"{a}hler potential $u$ (when $u$ solves the Monge-Amp\`{e}re equation) is
invariant whether $u$ is considered as a scalar function or a density. To
compute the nonzero right-hand term in (\ref{eq.med}) we first note that if
$f$ is a holomorphic map then we have the identity%
\[
\left(  v\circ f\right)  _{j\overline{k}}=\left[  \left(  f^{\prime}\right)
^{t}\left(  v_{a\overline{b}}\circ f\right)  \left(  \overline{f^{\prime}%
}\right)  \right]  _{j\overline{k}}.
\]
Thus, using (\ref{eq.J.alt}), we compute
\begin{align*}
J\left[  \left\vert \det (f^{\prime})\right\vert ^{-2/(m+1)}u\circ f\right]   &
=\left\vert \det (f^{\prime})\right\vert ^{-2}\left(  u\circ f\right)  ^{m+1}\\
&  \times\det\left(  \left(  f^{\prime}\right)  ^{t}\right)  \det\left(
\log\left(  \frac{1}{u}\right)  _{a\overline{b}}\circ f\right)  \det\left(
\overline{f^{\prime}}\right)  \\
&  =\left(  u\circ f\right)  ^{m+1}\det\left(  \log\left(  \frac{1}{u}\right)
_{a\overline{b}}\right)  \circ f.\\
&  =J\left[  u\right]  \circ f
\end{align*}
as was to be proved.

It is essential for our globalization argument that an approximate
solution to the Monge-Ampere equation be determined uniquely up to a
certain order. This proof was given by
Fefferman \cite{Fefferman:1976} and we repeat it for the reader's convenience.

\begin{lemma}
Any smooth, local, approximate solution $\psi \in \mathcal{C}^\infty (\Omega)$
to the Monge-Amp\`ere equation is uniquely
determined up to order $m+1$.
\end{lemma}

\begin{proof}
Suppose that $\rho$ is a smooth function on $\Omega$ defined in a neighborhood
of $\partial \Omega$ with $\rho=0$ on $\partial\Omega$ and $\rho^{\prime}(p)
\neq 0$ for all $p\in\partial\Omega$. We recall Fefferman's
iterative construction of an approximate solution $u$ to the
Monge-Amp\`{e}re equation, i.e., a function
$u\in\mathcal{C}^{\infty}$ with $u | \partial \Omega = 0$ and $J[u]=1+\mathcal{O}(u^{m+2})$.
To obtain a first approximation, note that for $\rho$ as above, and for any smooth
function $\eta$, we have
\begin{equation}
J\left[  \eta\rho\right ]  =\eta^{m+1}J [\rho], \label{eq.J11}%
\end{equation}
when $\rho=0$, so that the function
\[
\psi^{(1)}=\rho \cdot J[\rho]^{-1/(m+1)}%
\]
satisfies $J[ \psi^{(1)}] = 1$ on $\partial \Omega$,
and $J[ \psi^{(1)} ]=1+\mathcal{O}(\psi^{(1)})$. The fact that $J[\rho]$ is nonzero on
$\partial\Omega$ follows from pseudoconvexity that implies
that $\rho_{j \overline{k}}$ is positive definite
on $\mbox{ker} ~\partial \rho$ on $\partial \Omega$, and that
$\rho' \neq 0$ on $\partial \Omega$. Note that if $\varphi$ and $\psi$ are
two functions vanishing on $\partial\Omega$, it follows that $\varphi=\eta
\psi$ for some smooth function $\eta$. Thus, by (\ref{eq.J11}),
$J[\varphi] = \eta^{m+1} J[\psi]$. From this computation it follows that any approximate
solution $u$ is uniquely determined up to first order.

We now iterate this construction. Suppose that for an integer $s \geq 2$, we have
an approximate solution to the Monge-Amp\`ere equation to order $s-1$.
That is, we have a smooth function $\psi$
with $\psi=0$ on $\partial\Omega$, $\psi^{\prime} (p) \neq0$, for all $p\in
\partial\Omega$, and $J[\psi]=1+\mathcal{O}(\psi^{s-1})$.
We seek a function of the form $v=\psi+\eta\psi^{s}$,
where $\eta\in\mathcal{C}^{\infty}$ is chosen so that $J[v]=1+\mathcal{O}(\psi^{s})$.
The iteration is based on formula
\[
J[ \psi+\eta\psi^{s}] = J[ \psi ] + \left[  s(m+2-s)\right]  \eta\psi^{s-1}%
+\mathcal{O}(\psi^{s}),
\]
for smooth functions $\psi$ and $\eta$, again with the property that $\psi$ vanishes
on $\partial\Omega$. This formula is a straightforward computation using the
formula (\ref{eq.J}). From this formula it follows that the desired function
$v$ is given by%
\[
v=\psi+\left[  \frac{1-J(\psi)}{s(m+2-s)}\right]  \psi^{s} . %
\]
The iteration clearly works up to $s=m+1$ and produces an approximate solution
with the desired properties. It also follows that any function $\widetilde{u}$
with $u-\widetilde{u}=\mathcal{O}(\psi^{m+2})$ satisfies $J[ \widetilde
{u} ] = J[u] + \mathcal{O}(\psi^{m+2})$.\ Thus, in particular, any smooth function
having the same $(m+1)$-jet on $\partial\Omega$ as an approximate solution is
also an approximate solution.

On the other hand, it is clear that any two approximate solutions \emph{must}
have the same $(m+1)$-jet on $\partial\Omega$. If $\psi$ and $\widetilde{\psi
}$ satisfy $\psi-\widetilde{\psi}=\eta\psi^{s}$ then $J[ \psi ] - J [ \widetilde
{\psi}] = s\left(  m+2-s\right)  \eta\psi^{s-1}+\mathcal{O}(\psi^{s})$. In
particular, if $s<m+2$ and $J[\psi]-J[\widetilde{\psi}]=\mathcal{O}(\psi
^{m+2})$ then $\psi$ and $\widetilde{\psi}$ are approximate solutions
uniquely determined up to
order $m+2$.

\end{proof}

\subsubsection{Global Theory}

\label{sec.monge-ampere.global}

Now suppose $X$ is a compact complex manifold of dimension $m=n+1$ with
boundary $M=\partial X$. Note that $M$ has real dimension $2n+1$ and inherits
an integrable CR-structure from $X$. As always, we assume that $M$ with this
CR-structure is strictly pseudoconvex. We first say what it means for a single
smooth function $\varphi$ defined in a neighborhood of $M$ to be an
approximate solution of the complex Monge-Amp\`{e}re equation. We denote by
$\mathcal{C}^{\infty}(X)$ the smooth functions on $X$.


\begin{definition}
\label{def.gdas} We will say that a function $\varphi\in\mathcal{C}^{\infty
}(X)$ is a globally defined approximate solution of the complex
Monge-Amp\`{e}re equation near $M=\partial X$ if for any $p\in M$, there is a
neighborhood $V$ of $p$ in $X$ and holomorphic coordinates $z$ on $V$ so that
$\varphi$ is an approximate solution of the complex Monge-Amp\`{e}re equation
in the chosen coordinates.
\end{definition}

As we will see later, we will need such a globally defined approximate
solution in order to identify the residues of the scattering operator on $X$
with CR-covariant differential operators.

If $\varphi$ is a defining function for $M$ with $\varphi<0$ in the interior
of $X$, we associate to $\varphi$ a K\"{a}hler form%
\begin{equation}
\omega_{\varphi}=\frac{i}{2}\partial\overline{\partial}\log(-1/\varphi)
\label{eq.Kahler.form}%
\end{equation}
and a pseudo-Hermitian structure
\begin{equation}
\theta=\left.  \frac{i}{2}\left(  \overline{\partial}-\partial\right)
\varphi\right\vert _{M}. \label{eq.contact}%
\end{equation}
Observe that two defining functions $\varphi$ and $\rho$ generate the same
K\"{a}hler metric if and only if $\rho=e^{F}\varphi$ for a pluriharmonic
function $F$, i.e. $\partial\overline{\partial}F=0$. It is known that a
pluriharmonic function $F$ is uniquely determined by its boundary values (see,
for example, Bedford \cite{Bedford:1980}). If $\theta_{\rho}$ and
$\theta_{\varphi}$ are the corresponding pseudo-Hermitian structures on $M$
then $\theta_{\rho}=e^{f}\theta_{\varphi}$, where $f=\left.  F\right|  _{M}$.

We give a necessary and sufficient condition on $M$ for a globally defined
approximate solution of the Monge-Amp\`{e}re equation to exist.
Recall that the canonical bundle of $M$ is the bundle
generated by forms $f~\theta^{1}\wedge\cdots\wedge\theta^{n}\wedge\theta$
where $f$ is smooth, $\theta$ is a contact form, and $\left\{  \theta^{\alpha
}\right\}  _{\alpha=1}^{n}$ is an admissible co-frame. If $M$ is the boundary
of a strictly pseudoconvex domain in $\mathbb{C}^{m}$, the canonical bundle is
generated by restrictions of forms $f~dz^{1}\wedge\cdots\wedge dz^{m}$ to $M$.
The sections of the canonical bundle are $(n+1,0)$-forms $\zeta$ on $M$.

If $\theta$ is a contact form, $T$ is the characteristic vector field, and
$\zeta$ is \emph{any} nonvanishing section of the canonical bundle, it is not
difficult to see that
\begin{align*}
\theta\wedge(T~\lrcorner~\zeta) \wedge(T ~ \lrcorner~ \overline{\zeta}  &  =
\lambda ~\theta\wedge(d\theta)^{n}%
\end{align*}
for a smooth positive function $\lambda$. We say that the contact form
$\theta$ is \emph{volume-normalized} with respect to a nonvanishing section
$\zeta$ of the canonical bundle if
\[
\theta\wedge\left(  d\theta\right)  ^{n}=\left(  i\right)  ^{n^{2}}%
n!\theta\wedge\left(  T~\lrcorner~\zeta\right)  \wedge\left(  T~\lrcorner
~\overline{\zeta}\right)
\]
where $T$ is the characteristic vector field. The following criterion will be useful.

\begin{lemma}
\label{lemma.local}The contact form $\theta$ given by (\ref{eq.contact}) is
volume-normalized with respect to the form $\zeta=\left.  dz^{1}\wedge
\cdots\wedge dz^{m}\right\vert _{M}$ if and only if
\[
J [\varphi]=1+\mathcal{O}(\varphi)
\]
in the coordinates $(z_{1},\cdots,z_{m})$.
\end{lemma}

For the proof see Farris \cite{Farris:1986}, Proposition 5.2. Using the lemma,
we can prove:

\begin{proposition}
\label{prop.ma.global}Suppose that $X$ is a compact complex manifold with
boundary $M=\partial X$. There is a globally defined approximate solution
$\varphi$ of the Monge-Amp\`{e}re equation in a neighborhood of $M$ if and
only if $M$ admits a pseudo-Hermitian structure $\theta$ with the following property:
In a neighborhood of any point $p \in M$, there is a local, closed $(n+1,0)$ form
$\zeta$ such that $\theta$ is volume-normalized with respect to $\zeta$.
\end{proposition}

\begin{proof}
(i) First, suppose that $X$ admits a globally defined approximate solution
$\varphi$ of the Monge-Amp\`{e}re equation. Let $\theta$ be the associated
contact form on $X$, i.e., $\theta$ is given by (\ref{eq.contact}). Pick $p\in
M$ and let $z \equiv \left(  z_{1},\cdots,z_{m}\right)  $ be holomorphic coordinates
near $p$ so chosen that $\varphi$ is an approximate solution of the
Monge-Amp\`{e}re equation near $p$ in these coordinates. Let
\[
\zeta=\left.  dz^{1}\wedge\cdots\wedge dz^{m}\right\vert _{M}.
\]
Then $\theta$ is volume-normalized with respect to $\zeta$ by Lemma
\ref{lemma.local}.

(ii) Suppose that $\theta$ is a given contact form on $M$ with the property
that, for each point $p\in M$, there is a neighborhood of $p$ and a closed,
locally defined section $\zeta$ of the canonical bundle with respect to which
$\theta$ is volume-normalized. Write
\[
\zeta=f~\left.  dz^{1}\wedge\cdots\wedge dz^{m}\right\vert _{M}%
\]
for holomorphic coordinates $\left\{  z_{1},\cdots,z_{m}\right\}  $ defined in
a neighborhood of $p$ and a smooth function $f$. The condition $d\zeta=0$ is
equivalent to the condition%
\[
\overline{\partial_{b}}f=0
\]
i.e., $\ f$ is a CR-holomorphic function. By the strict pseudoconvexity of
$M$, there is a holomorphic extension $F$ to a neighborhood $V$ of $p$ in $X$,
i.e., there is an $F$ defined near $p$ with $\overline{\partial}F=0$ and
$\left.  F\right\vert _{M\cap V}=f$ (see \cite{KR:1965}). We claim that we can
find new holomorphic coordinates $w \equiv (  w_{1},\cdots,w_{m} )$ near
$p$ with the property that
\begin{equation}
\frac{\partial\left(  w_{1},\cdots,w_{m}\right)  }{\partial\left(
z_{1},\cdots,z_{m}\right)  }=F (z) \label{eq.Jac.F}%
\end{equation}
If so then%
\[
\zeta=\left.  dw^{1}\wedge\cdots\wedge dw^{m}\right\vert _{M}%
\]
Constructing in $V$ an approximate solution $\psi_{V}$ of the
Monge-Amp\`{e}re equation in the $w$-coordinates (as in Lemma 2.7, following Fefferman
\cite{Fefferman:1976}), we conclude from Lemma \ref{lemma.local}
that the induced contact form
\[
\theta_{V}=\left.  \frac{i}{2}\left(  \overline{\partial}-\partial\right)
\psi_{V}\right\vert _{M\cap V}%
\]
on $M\cap V$ is volume-normalized with respect to $\zeta$, and thus coincides
with $\theta$.

We now claim that the local approximate solutions $\psi_{V}$ can be
glued together to form a globally defined approximate solution to
the Monge-Amp\`{e}re equation in the sense of Definition
\ref{def.gdas}. We first note an important property of the
transition map for two local coordinates. Let $V_{1}$ and $V_{2}$ be
neighborhoods of $M$ in $X$ with nonempty intersection, let $z$ and
$w$ be holomorphic coordinates on $V_{1}$ and $V_{2}$, and suppose
that $\psi_{1}$ and $\psi_{2}$ are approximate solutions of the
complex Monge-Amp\`{e}re equation in these respective coordinates.
More precisely, $u_{1}=\psi_{1}\circ z$ and $u_{2}=\psi_{2}\circ w$
are approximate solutions to the Monge-Amp\`{e}re equation on
coordinate patches $U_{1}$ and $U_{2}$ in $\mathbb{C}^{m}$, and
there is a biholomorphic map $g:U_{2}\cap w^{-1}(V_{1}\cap
V_{2})\rightarrow U_{1}\cap z^{-1}(V_{1}\cap
V_{2})$. The function $u_{2}=\left\vert g^{\prime}\right\vert ^{2/(m+1)}%
u_{1}\circ g$ is also an approximate solution of the complex
Monge-Ampere equation in $U_{2}\cap w^{-1}(V_{1}\cap V_{2})$ by
Lemma \ref{lemma.f}, so by uniqueness we have $u_{2}=e^{F}u_{1}\circ
g$, up to order $m+1$, where $F=(2/(m+1))\log\left\vert
g^{\prime}\right\vert $ is pluriharmonic. Moreover, since
$u_{1}$ and $u_{2}$ both induce the contact form
$\theta$ it follows that
\[
\left.  \left(  \overline{\partial}-\partial\right)  u_{2}\right\vert
_{U_{2}\cap w^{-1}(M\cap V_{1}\cap V_{2})}=\left.  \left[  \left(
\overline{\partial}-\partial\right)  u_{1}\right]  \circ f\right\vert
_{U_{2}\cap w^{-1}(M\cap V_{1}\cap V_{2})}%
\]
from which we deduce that $\left.  F\right\vert _{U_{2}\cap
w^{-1}(M\cap V_{1}\cap V_{2})}=0$, and hence $F=0$ by the uniqueness
of pluriharmonic extensions. In particular, the map $g$ is
unimodular, $|g'| = 1$. Thus $u_{2}=u_{1}\circ g$ on $U_{2}\cap
w^{-1}(V_{1}\cap V_{2})$ up to order $m+1$.

We now fix a boundary defining function $\rho$. Suppose that
$\left\{ U_{i}\right\} $ is a finite cover of a neighborhood of the
boundary by holomorphic charts. Denote by $F_{i}$ the map from
$\mathbb{C}^{m}$ into $U_{i}$ and set $F_{ij}=F_{i}^{-1}\circ
F_{j}$. As proved above, the cover and holomorphic coordinates
$\left( U_{i},F_{i}\right) $ may be chosen so that the transition
maps are unimodular, i.e., $\left\vert F_{ij}^{\prime}\right\vert
=1$.
Using Fefferman's construction, we can produce in each $U_{i}$ an
approximate solution $u_{i}$ in the sense that
\[
J\left[  u_{i}\right]  =1+\mathcal{O}\mathbf{(}\rho^{m+1})
\]
Now suppose that $\left\{  \chi_{i}\right\}  $ is a
$\mathcal{C}^{\infty}$ partition of unity subordinate to the cover
$\left\{  U_{i}\right\}  $. We claim that the smooth function
$u=\sum_{i}\chi_{i}u_{i}$ is an approximate solution of the
Monge-Amp\`{e}re equation in the sense of Definition \ref{def.gdas}.
Choose $U_{i}$ so that $p\in U$. We may write
$u=\sum_{j}(\chi_{j}\circ F_{i})(z)~\left(  u_{j}\circ
F_{i}\right)  $. Since $u_{j}\circ F_{i}=(u_{j}\circ F_{j})\circ
F_{ji}$ we see that $( u_j \circ F_i )$ is also
an approximate solution to the Monge-Amp\`{e}re equation in the $F_{i}%
$-coordinates. Thus, there is a smooth function $\eta_{ji}$ so that %
\[
\left(  u_{j}\circ F_{i}\right)  (z)-\left(  u_{i}\circ F_{i}\right)
(z)=\eta_{ji}(z)\left(  \rho\circ F_{i}\right)^{m+2}(z)%
\]
where $\eta_{ji}$ is smooth. We conclude that%
\[
u(z) - u_i (z) =
\mathcal{O} ( ( \rho \circ F_i)^{m+2}).
\]
This shows that $u$ is also an approximate solution of the
Monge-Amp\`{e}re equation in the $F_{i}$-coordinates as claimed.



To finish the proof it suffices to establish that such a holomorphic
coordinate change $z\mapsto w$, as in (\ref{eq.Jac.F}), exists.
We consider a coordinate transformation given by
\begin{equation}\label{coord-2}
w(z) = ( h(z), z_2, \ldots , z_m),
\end{equation}
where $h(z)$ is the unknown holomorphic function. Condition (\ref{eq.Jac.F})
is equivalent to
\begin{equation}\label{eq.cond}
\frac{\partial h}{\partial z_1} (z_1, \ldots, z_m )
 =  F(z_1, z_2, \ldots, z_m ) .
\end{equation}
Here, $F$ is the holomorphic  extension of the CR-function $f$.
We solve this equation for $h$ as follows.
%
%
%
%
%
We set the convention that a boundary chart in $\mathbb{C}^{m}$ is
the intersection of an open ball about $0$ with the (real)\
half-space $\operatorname{Im}z_{m}\geq0$.
We assume that the boundary point $p$ corresponds to $0 \in \partial
\mathbb{C}^m$.
%
%
The unknown function
$h$ is a complex-valued function defined in a neighborhood $V$
of
$0\in\mathbb{C}^{m}$, is holomorphic in $V\cap\left\{  \operatorname{Im}%
z_{m}>0\right\}  $, has CR boundary values, and satisfies $h(0)=0$.
Thus, the map $w(z)$, defined in (\ref{coord-2}), preserves
the boundary $\operatorname{Im}(z_{n})=0$.

Consequently, the desired change of coordinates is obtained by solving the initial value problem%
\begin{align}
\frac{ \partial h }{\partial z_{1} }(z_{1},\cdots,z_{m}) &  = F (z_1,z_{2},\cdots
,z_{m})\label{eq.code}\\
h(0,z_{2},\cdots,z_{m}) &  =0 ,\nonumber
\end{align}
by simple integration.
\end{proof}



We can also express the basic criterion in Proposition \ref{prop.ma.global} in geometric terms.
Recall that the
contact form $\theta$ defines a pseudo-Hermitian, pseudo-Einstein structure on
$M$ if the Webster Ricci tensor is a multiple of the Levi form. Lee
\cite{Lee:1988} proved:

\begin{theorem}
\label{thm.pseudo.lee}Suppose that $M$ is a CR-manifold of dimension $\geq5$.
A contact form $\theta$ on $M$ is pseudo-Einstein if and only if for each
$p\in M$ there is a neighborhood of $p$ in $M$ and a locally defined closed
section $\zeta$ of the canonical bundle with respect to which $\theta$ is volume-normalized.
\end{theorem}

As an immediate consequence of Theorem \ref{thm.pseudo.lee}, we have:

\begin{theorem}
\label{thm.pseudo}Suppose that $M$ is a CR-manifold of dimension $\geq5$.
There is a globally defined approximate solution $\varphi$ of the complex
Monge-Amp\`{e}re equation in a neighborhood of $M$ if and only if $M$ carries
a contact form $\theta$ for which the corresponding pseudo-Hermitian structure
is pseudo-Einstein. In this case, the contact form $\theta$ is induced by the globally defined
approximate solution to the Monge-Amp\`ere equation $\varphi$.
\end{theorem}

\begin{remark}
If $\varphi$ is a global approximate solution to the Monge-Amp\`{e}re
equation, then so is $e^{F}\varphi$ where $F$ is any pluriharmonic function.
The effect of the factor $F$ is simply to change the choice of local
coordinates needed to obtain a local approximate solution of the
Monge-Amp\`{e}re equation in any chart, as the argument in the proof of
Proposition \ref{prop.ma.global} easily shows. As observed above, the
K\"{a}hler form $\omega_{\varphi}$ is invariant under the change
$\varphi\mapsto e^{F}\varphi$.
\end{remark}

\section{Poisson Operator and Scattering Operator}

In this section we study the Dirichlet problem (\ref{eq.dp}) following a
standard technique in geometric scattering theory (see, for example, Melrose
\cite{Melrose:1995}; we follow closely the analysis of the Poisson operator
and scattering operator on conformally compact manifolds by Graham and Zworski
in \cite{GZ:2003}). Note that Epstein, Melrose, and Mendoza \cite{EMM:1991}
had previously studied the Poisson operator for a class of manifolds that
includes compact complex manifolds with strictly pseudoconvex boundaries. More
recently, Guillarmou and S\'a Barreto \cite{GSb:2006} studied scattering
theory and radiation fields for asymptotically complex hyperbolic manifolds, a
class which also includes that studied here.

We will set $x=-\varphi$ and we will denote by $\mathcal{C}^{\infty}(X)$ the
set of smooth functions on $X~$having Taylor series to all orders at $x=0$,
and by $\mathcal{\dot{C}}^{\infty}(X)$ the space of functions vanishing to all
orders at $x=0$. The space $\mathcal{C}^{\infty}(\mathring{X})$ consists of
smooth functions on $\mathring{X}$ with no restriction on boundary behavior.
We will denote by $x^{s}\mathcal{C}^{\infty}(X)$ the set of functions in
$\mathcal{C}^{\infty}(\mathring{X})$ having the form $x^{s}F$ for
$F\in\mathcal{C}^{\infty}(X)$.

Since
\[
N=-2\frac{\partial}{\partial x}%
\]
it follows that
\begin{equation}
\nu=-\frac{x}{\sqrt{1+rx}}\frac{\partial}{\partial x} \label{eq.nu.2}%
\end{equation}
is the outward normal to the hypersurface $x=\varepsilon$. Green's theorem
implies that%
\begin{equation}
\int_{x>\varepsilon}\left(  u_{1}\Delta_{\varphi}u_{2}-u_{2}\Delta_{\varphi
}u_{1}\right)  ~\omega^{m}=\int_{x=\varepsilon}\left(  u_{1}\nu u_{2}-u_{2}\nu
u_{1}\right)  ~\nu~\lrcorner~\omega^{m} \label{eq.Green}%
\end{equation}
We first note the `boundary pairing formula' (recall the definition
(\ref{eq.psi})).

\begin{proposition}
\label{prop.bp}Suppose $\operatorname{Re}(s)=m/2$, that $u_{1}$ and $u_{2}$
belong to $\mathcal{C}^{\infty}(\mathring{X})$ and there are functions
$F_{i},G_{i}\in\mathcal{C}^{\infty}(X)$ so that $u_{i}=x^{m-s}F_{i}+x^{s}%
G_{i}$, $i=1,2$. Finally, suppose that $\left(  \Delta_{\varphi}%
-s(m-s)\right)  u_{i}=r_{i}\in\mathcal{\dot{C}}^{\infty}(X)$, $i=1,2$. Then,
the formula%
\begin{equation}
\int_{X}\left(  u_{1}r_{2}-u_{2}r_{1}\right)  ~\omega^{m}=\left(  2s-m\right)
\int_{M}\left(  F_{1}G_{2}-F_{2}G_{1}\right)  ~\psi\label{eq.bp.1}%
\end{equation}
holds.
\end{proposition}

\begin{proof}
A standard computation using (\ref{eq.Green}) and (\ref{eq.nu.2}) together
with (\ref{eq.surface}) and (\ref{eq.lap}).
\end{proof}

\begin{remark}
For $\operatorname{Re}(s)=m/2$ complex conjugation reverses the roles of $s$
and $m-s$. Thus we obtain the formula%
\begin{equation}
\int_{X}\left(  u_{1}\overline{r_{2}}-\overline{u_{2}}r_{1}\right)
~\omega^{m}=(2s-m)\int_{M}\left(  F_{1}\overline{F_{2}}-G_{1}\overline{G_{2}%
}\right)  ~\psi\label{eq.bp.2}%
\end{equation}

\end{remark}

For later use, we note an extension of the boundary pairing formula analogous
to Proposition 3.3 of \cite{GZ:2003}.

\begin{proposition}
\label{prop.bp.extended}Suppose that $\operatorname{Re}(s)>m/2$ and
$2s-m\notin\mathbb{N}$. Suppose that $u_{i}\in\mathcal{C}^{\infty}%
(\mathring{X})$ takes the form%
\begin{equation}
u_{i}=x^{m-s}F_{i}+x^{s}G_{i} \label{eq.u.form}%
\end{equation}
(\ref{eq.u.form}) and $\left(  \Delta_{\varphi}-s(m-s)\right)  u_{i}=0$,
$i=1,2$. Then%
\[
\operatorname*{FP}_{\varepsilon\downarrow0}\left(  \int_{x>\varepsilon}\left[
\left\langle \nabla u_{1},\nabla u_{2}\right\rangle -s(m-s)u_{1}u_{2}\right]
~\omega^{m}\right)  =-m\int_{M}G_{1}F_{2}~\psi=-m\int_{M}F_{1}G_{2}~\psi
\]
where $\operatorname*{FP}$ denotes the Hadamard finite part of the integral as
$\varepsilon\downarrow0$.
\end{proposition}

\begin{proof}
Green's formula (\ref{eq.Green}) for the operator $\Delta_{\varphi}-s(m-s)$
gives%
\[
\int_{x>\varepsilon}\left[  \left\langle \nabla u_{1},\nabla u_{2}%
\right\rangle -s(m-s)u_{1}u_{2}\right]  \omega^{m}=\int_{x=\varepsilon}%
u_{1}\left(  \nu u_{2}\right)  ~\nu~\lrcorner~\omega^{m}%
\]
from which the claimed formulae follow.
\end{proof}

\subsection{The Poisson Map}

We now prove that the Dirichlet problem (\ref{eq.dp}) has a unique solution if
$\operatorname{Re}(s)\geq m/2$, $2s-m\notin\mathbb{Z}$, and $s(m-s)$ is not an
eigenvalue of $\Delta_{\varphi}$. Most of the formal arguments are almost
identical to the case of even asymptotically hyperbolic manifolds considered
in \cite{GZ:2003} since the form of the normal operator (\ref{eq.normal}) for
the Laplacian is the same.

\begin{lemma}
\label{lemma.formal}Suppose that $u\in\mathcal{C}^{\infty}(\mathring{X})$
satisfies $u=x^{m-s}F+x^{s}G$ for functions $F$ and $G$ belonging to
$\mathcal{C}^{\infty}(X)$, and that%
\begin{equation}
\left(  \Delta_{\varphi}-s(m-s)\right)  u\in\mathcal{\dot{C}}^{\infty}(X)
\label{eq.approx.ef}%
\end{equation}
for $s\in\mathbb{C}$ with $2s-m\notin\mathbb{Z}$. Then the Taylor expansions
of $F$ and $G$ at $x=0$ are formally determined respectively by $\left.
F\right\vert _{M}$ and $\left.  G\right\vert _{M}$. In particular, we have
$F\sim\sum_{k\geq0}x^{k}f_{k}$ and $G\sim\sum_{k\geq0}x^{k}g_{k}$ where%
\begin{equation}
f_{k}=\frac{1}{k!}\frac{\Gamma(2s-m-k)}{\Gamma(2s-m)}P_{k,s}f_{0}
\label{eq.Pks}%
\end{equation}
and%
\[
g_{k}=\frac{1}{k!}\frac{\Gamma(m-2s-k)}{\Gamma(m-2s)}P_{k,m-s}g_{0}%
\]
where $P_{k,s}$ are differential operators of order $2k$ holomorphic in $s$
with leading symbol\footnote{Here in the sense of the ordinary (rather than
the Heisenberg) calculus on $M$.}%
\[
\sigma(P_{k,s})=\frac{1}{2^{k}}\sigma\left(  -\Delta_{b}^{k}\right)
\]

\end{lemma}

\begin{proof}
Recall the asymptotic development (\ref{eq.Deltag.exp}) for the Laplacian
which we use to derive a recurrence for the Taylor coefficients $f_{k}$ and
$g_{k}$ of $F$ and $G$. For $2s-m\notin\mathbb{Z}$, we may consider the terms
involving $F$ and $G$ separately. We first consider $F$. Observe that
\[
(L_{0}-s(m-s))\left(  x^{m-s+k}f\right)  =k(2s-m-k)x^{s+k}f
\]
for $f\in\mathcal{C}^{\infty}(M)$. Since $L_{k}=P(x\partial_{x},\partial_{y})$
for a defining function $x$ and boundary coordinates $y$ where $P$ is a
polynomial of degree at most two with smooth coefficients, the operators%
\[
Q_{k,\ell}(s)=x^{-m+s-\ell}L_{k-\ell}x^{m-s+\ell}%
\]
are differential operators of order at most two depending holomorphically on
$s$. If $u\sim\sum_{k=0}^{\infty}x^{m-s+k}f_{k}$, it follows from
(\ref{eq.approx.ef}) and (\ref{eq.Deltag.exp}) that for any $k\geq1$,
\begin{equation}
f_{k}=-\frac{1}{k(2s-m-k)}\sum_{\ell=0}^{k-1}Q_{k,\ell}(s)f_{\ell}
\label{eq.fk}%
\end{equation}
Similarly, if $u\sim\sum_{k\geq0}x^{s+k}g_{k}$ for $g_{k}\in\mathcal{C}%
^{\infty}(M)$, we have%
\begin{equation}
g_{k}=-\frac{1}{k(m-2s-k)}\sum_{\ell=0}^{k-1}Q_{k,\ell}(m-s)g_{\ell}
\label{eq.gk}%
\end{equation}
The formulas for $f_{k}$, $g_{k}$, and $P_{k,s}$ follow easily from these
formulas and the fact that%
\[
Q_{k,k-1}(s)=\frac{1}{4}\left(  -2\Delta_{b}u-4r_{0}\left(  m-s+1\right)
-4r_{0}\left[  \left(  m-s+1\right)  ^{2}-\left(  m-s+1\right)  \right]
\right)
\]

\end{proof}

\begin{remark}
\label{rem.pl}We will write $p_{k,s}$ for the operator with $f_{k}%
=p_{k,s}f_{0}$, so that $p_{k,s}$ is meromorphic with poles at
$s=m/2+k/2,\ldots,m/2+1/2$. We will denote%
\[
p_{\ell}=\operatorname*{Res}_{s=m/2+\ell/2}p_{\ell,s}%
\]
The operator $p_{\ell}$ is a differential operator of order at most $2\ell$
with principal symbol
\[
\sigma(p_{\ell})=\frac{1}{2^{\ell}\ell!(\ell-1)!}\sigma\left(  -\Delta
_{b}^{\ell}\right)
\]

\end{remark}

For $\operatorname{Re}(s)>m/2$, let
\[
R(s)=\left(  \Delta_{\varphi}-s(m-s)\right)  ^{-1}%
\]
be the $L^{2}(X)$ resolvent, let $\sigma_{p}(\Delta_{\varphi})$ denote the set
of $L^{2}$-eigenvalues of $\Delta_{\varphi}$, and let%
\[
\Sigma=\left\{  s:\operatorname{Re}(s)>m/2,~s(m-s)\in\sigma_{p}(\Delta
_{\varphi})\right\}  .
\]
We will now solve the Dirichlet problem (\ref{eq.dp}) for $\operatorname{Re}%
(s)\geq m/2$ and $s\notin\Sigma$.

The following result is an easy consequence of the work of Epstein, Melrose,
and Mendoza \cite{EMM:1991}, noting that in our case the K\"{a}hler metric is
an even metric, i.e., depends smoothly on the defining function $\varphi$ (and
not simply on its square root).

\begin{proposition}
\label{prop.emm}The set $\Sigma$ contains at most finitely many points, and
the resolvent operator $R(s)$ is a meromorphic operator-valued function for
$\operatorname{Re}(s)>m/2-1/2$ having at most finitely many, finite-rank poles
at $s\in\Sigma$. Moreover, for $s\notin\Sigma$,and $\operatorname{Re}%
(s)>m/2-1/2$, $R(s):\mathcal{\dot{C}}^{\infty}(X)\rightarrow x^{s}%
\mathcal{C}^{\infty}(X)$.
\end{proposition}

First, we prove uniqueness of solutions to the Dirichlet problem (\ref{eq.dp})
for $s$ with $\operatorname{Re}(s)\geq m/2$, $s\notin\Sigma$, and
$2s-m\notin\mathbb{Z}$.

\begin{proposition}
Suppose that $\operatorname{Re}(s)\geq m/2$, $s\notin\Sigma$, and
$2s-m\notin\mathbb{Z}$. Suppose that $u\in\mathcal{C}^{\infty}(\mathring{X})$
with $\left(  \Delta_{\varphi}-s(m-s)\right)  u=0$, and that $u=x^{m-s}%
F+x^{s}G$ with $\left.  F\right\vert _{M}=0$. Then $u=0$.
\end{proposition}

\begin{proof}
First, suppose that $\operatorname{Re}(s)>m/2$ and $s\notin\Sigma$. It follows
from Lemma \ref{lemma.formal} that $u=x^{s}G$ for $G\in\mathcal{C}^{\infty
}(X)$. Since $\operatorname{Re}(s)>m/2$ it is clear that $\int_{X}\left\vert
u\right\vert ^{2}~\omega^{m}<\infty$, hence $u\in L^{2}(X)$, hence $u=0$.

If $\operatorname{Re}(s)=m/2$ but $s\neq m/2$, we may again assume that
$u=x^{s}G$ for $G\in\mathcal{C}^{\infty}(X)$. Next, we set $u_{1}=u_{2}=u$ in
(\ref{eq.bp.2}) to conclude that $\int_{M}\left\vert G\right\vert ^{2}\psi=0$
so that $\left.  G\right\vert _{M}=0$. Using Lemma \ref{lemma.formal} again we
conclude that $G\in\mathcal{\dot{C}}^{\infty}(X)$, hence $u\in\mathcal{\dot
{C}}^{\infty}(X)$. As in \cite{GSb:2006}, we can now deduce from
\cite{VW:2004} that $u=0$.
\end{proof}

To prove existence of a solution of the Dirichlet problem (\ref{eq.dp}), we
follow the method of Graham and Zworski \cite{GZ:2003}. Given $f\in
\mathcal{C}^{\infty}(M)$ we can construct a formal power series solution
$u=x^{n-s}F$ modulo $\mathcal{\dot{C}}^{\infty}(X)$, and then use the
resolvent to correct this approximate solution to an exact solution. Using
Borel's lemma we can sum the asymptotic series $\sum_{j\geq0}f_{j}x^{j}$
(where $f_{j}$ is computed via (\ref{eq.fk}) with $f_{0}=f$) to a function
$F\in\mathcal{C}^{\infty}(X)$. As in \cite{GZ:2003}, we obtain:

\begin{lemma}
There is an operator $\Phi(s):\mathcal{C}^{\infty}(M)\rightarrow
x^{n-s}\mathcal{C}^{\infty}(X)$ with
\[
\left(  \Delta_{\varphi}-s(m-s)\right)  \circ\Phi:\mathcal{C}^{\infty
}(M)\rightarrow\mathcal{\dot{C}}^{\infty}(X)
\]
so that $\Gamma(m-2s)^{-1}\Phi(s)$ is holomorphic in $s$.
\end{lemma}

Note that $\Phi(s)$ need not be linear as the construction of $F$ depends on
the choice of cutoff functions in the application of Borel's lemma. Now define
an operator%
\[
\mathcal{P}(s):\mathcal{C}^{\infty}(M)\rightarrow\mathcal{C}^{\infty
}(\mathring{X})
\]
for $s$ with $\operatorname{Re}(s)\geq m/2$, $s\neq m/2$ and $s \not \in
\Sigma$ by%
\[
\mathcal{P}(s)=\left[  I-R(s)\left(  \Delta_{\varphi}-s(m-s)\right)  \right]
\circ\Phi(s)
\]

\begin{lemma}
For any $f\in\mathcal{C}^{\infty}(M)$, the function $u=\mathcal{P}(s)f$ solves
the Dirichlet problem (\ref{eq.dp}), and $f\mapsto\mathcal{P}(s)f$ is a linear operator.
\end{lemma}

\begin{proof}
The linearity of $\mathcal{P}(s)$ will follow from the unicity of the solution
to (\ref{eq.dp}). It is immediate from the definition that $\left(
\Delta_{\varphi}-s(m-s)\right)  u=0$, and from the mapping property in
Proposition \ref{prop.emm}, $u=x^{m-s}F+x^{s}G$ with $F=x^{s-m}\Phi(s)f$ and
$G=-x^{-s}R(s)\left[  \left(  \Delta_{\varphi}-s(m-s)\right)  \Phi(s)f\right]
$.
\end{proof}

We now have:

\begin{theorem}
\label{thm.unique}For $\operatorname{Re}(s)\geq m/2$, $2s-m\notin\mathbb{Z}$,
and $s\notin\Sigma$, there exists a unique solution of the Dirichlet problem
(\ref{eq.dp}).
\end{theorem}

\subsection{The Scattering Operator}

The \emph{scattering operator} for $\Delta_{\varphi}$ is the linear mapping%
\begin{align*}
S_{X}(s)  &  :\mathcal{C}^{\infty}(M)\rightarrow\mathcal{C}^{\infty}(M)\\
f  &  \mapsto\left.  G\right\vert _{M}%
\end{align*}
where $u=x^{m-s}F+x^{s}G$ solves (\ref{eq.dp}). It is well-defined by Theorem
\ref{thm.unique}.

The scattering operator has infinite-rank poles when $\operatorname{Re}%
(s)>m/2$ and $2s-m\in\mathbb{Z}$ owing to the crossing of indicial roots for
the normal operator $L_{0}$. At the exceptional points $s=m/2+k$ one expects
solutions of the eigenvalue equation $\left(  \Delta_{\varphi}-s(m-s)\right)
u=0$ having the form%
\[
u=x^{m/2-k}F+\left(  x^{m/2+k}\log x\right)  G
\]
In order to study the singularities of the scattering operator at these points
we modify the construction of the Poisson operator following the lines of
\cite{GZ:2003}, section 4.

Let $f_{1}$ and $f_{2}$ belong to \ $\mathcal{C}^{\infty}(M)$ and let $u_{1}$
and $u_{2}$ solve the corresponding Dirichlet problems for some $s$ with
$\operatorname{Re}(s)>m/2$ and $2s-m\notin\mathbb{N}$. Applying the
generalized boundary pairing formula (see Proposition \ref{prop.bp.extended})
to $u_{1}$ and $\overline{u_{2}}$ for $s$ real, we conclude that
\[
\int_{M}f_{1}\overline{S_{X}(s)f_{2}}~\psi=\int_{M}\left[  S_{X}%
(s)f_{1}\right]  \overline{f_{2}}~\psi
\]
so $S_{X}(s)$ is self-adjoint in the natural inner product on $\mathcal{C}%
^{\infty}(M)$.

Now we study the scattering operator near the exceptional points. The
arguments used here are exactly those of section 3 in \cite{GZ:2003} but we
summarize them here for the reader's convenience.

Recall the operators $p_{k,s}$ and $p_{\ell}$ defined in Remark \ref{rem.pl}.
First, we prove:

\begin{lemma}
At the points $s=m/2+\ell/2$, $\ell=1,2,\cdots$, $s\notin\Sigma$, the Poisson
map takes the form%
\[
\mathcal{P}(m/2+\ell/2)f=x^{m/2-\ell/2}F+\left(  x^{n/2+\ell/2}\log x\right)
G
\]
where%
\[
\left.  F\right\vert _{M}=f
\]
and%
\[
\left.  G\right\vert _{M}=-2p_{\ell}f
\]
where%
\begin{equation}
p_{\ell}=\operatorname*{Res}_{s=m/2+\ell/2}p_{\ell,s} \label{eq.pl}%
\end{equation}
is a differential operator of order $2\ell$ with%
\[
\sigma(p_{\ell})=\frac{1}{2^{\ell}\ell!\left(  \ell-1\right)  !}\sigma
(\Delta_{b}^{\ell})
\]

\end{lemma}

\begin{proof}
We first show that the Poisson map $\mathcal{P}(s)$ is also regular at
$s=m/2+\ell/2$, $\ell=1,2,\cdots$ so long as these points do not belong to
$\Sigma$. As in \cite{GZ:2003} we introduce the operator%
\begin{equation}
\Phi_{\ell}(s)=\Phi(s)-\Phi(m-s)\circ p_{\ell,s} \label{eq.phi.l}%
\end{equation}
where $p_{\ell,s}$ is a differential operator of order $2\ell$ defined in
Remark \ref{rem.pl}. Each of the right-hand terms in (\ref{eq.phi.l}) has at
most a first-order pole at $s=m/2+\ell/2$; the operators $p_{j.s}$ occurring
in the definition of $\Phi(s)$ have at most first-order poles, while
$\Phi(m-s)$ is analytic in $s$ for $\operatorname{Re}(s)>m/2$. For given
$f\in\mathcal{C}^{\infty}(M)$, we compute the residue of $\Phi_{\ell}(s)f$ at
$s=m/2+\ell/2$. First%
\[
\lim_{s\rightarrow m/2+\ell/2}\left(  s-\frac{m}{2}-\frac{\ell}{2}\right)
\Phi(s)f=x^{m/2+\ell/2}\operatorname*{Res}_{s=m/2+\ell/2}\left(  p_{\ell
,s}f\right)  +\mathcal{O}\left(  x^{m/2+\ell/2+1}\right)
\]
since the remaining terms in the asymptotic expansion for $\Phi(s)f$ are
holomorphic near $s=m/2+\ell/2$. Second,%
\[
\lim_{s\rightarrow m/2+\ell/2}\left(  s-\frac{m}{2}-\frac{\ell}{2}\right)
\Phi(m-s)\left(  p_{\ell,s}f\right)  =x^{m/2+\ell/2}\operatorname*{Res}%
_{s=m/2+\ell/2}\left(  p_{\ell,s}f\right)  +\mathcal{O}\left(  x^{m/2+\ell
/2+1}\right)  .
\]
It follows that%
\begin{equation}
\operatorname*{Res}_{s=m/2+\ell/2}\Phi_{\ell}(s)f=\mathcal{O}\left(
x^{m/2+\ell/2+1}\right)  \label{eq.res}%
\end{equation}
so that, by Lemma \ref{lemma.formal}, $\operatorname*{Res}_{s=m/2+\ell/2}%
\Phi_{\ell}(s)f\in\mathcal{\dot{C}}^{\infty}(X)$.

Now let us define%
\[
\mathcal{P}_{\ell}(s)=\left[  I-R(s)\left(  \Delta_{\varphi}-s(m-s)\right)
\right]  \circ\Phi_{\ell}(s)
\]
Clearly, $\mathcal{P}_{\ell}(s)$ is holomorphic in a deleted neighborhood of
$s=m/2+\ell/2$ (with at most a first-order pole at $s=m/2+\ell/2$) and maps
$\mathcal{C}^{\infty}(M)$ into $\mathcal{C}^{\infty}(\mathring{X})$. If
$s\notin\Sigma$, it follows from the definition of $\mathcal{P}_{\ell}(s)$,
equation (\ref{eq.res}), and Proposition \ref{prop.emm} that%
\[
\operatorname*{Res}_{s=m/2+\ell/2}\mathcal{P}_{\ell}(s)f\in x^{s}%
\mathcal{C}^{\infty}(X),
\]
hence the residue is an $L^{2}(X)$ function, and hence is zero. Thus
$\mathcal{P}_{\ell}(s)$ is holomorphic at $s=m/2+\ell/2$. It follows from the
uniqueness of solutions to the Dirichlet problem that $\mathcal{P}_{\ell
}(s)=\mathcal{P}(s)$ wherever the former is defined. Exactly as in
\cite{GZ:2003} we can compute $\mathcal{P}(m/2+\ell/2)f$ by using
\ $\mathcal{P}_{\ell}(s)$, the formula%
\[
\lim_{t\rightarrow0}\frac{x^{-t}-x^{t}}{t}=-2\log x
\]
and the fact that the $p_{k,s}$ have at most simple poles at $s=m/2+\ell/2$.
This computation shows that $\mathcal{P}(m/2+\ell/2)$ has the stated form.
\end{proof}

Next, we prove:

\begin{proposition}
\label{prop.S.poles}Suppose that $\Delta_{X}$ has no eigenvalues of the form
$s(m-s)$ with $s=m/2+\ell/2$, $\ell=1,2,\cdots$. Then, the scattering operator
$S_{X}(s)$ has a first-order pole at $s=m/2+\ell/2$, $\ell=1,2,\cdots$ with%
\[
\operatorname*{Res}_{s=m/2+\ell/2}S_{X}(s)=p_{\ell}\text{.}%
\]
where $p_{\ell}$ is the differential operator given by (\ref{eq.pl}).
\end{proposition}

\begin{proof}
From the formula for the $\mathcal{P}_{\ell}(s)$, it is clear that for
$2s-m\notin\mathbb{N}$, we can compute the scattering operator from
\[
S_{X}(s)f=\left.  \left[  -x^{-s}R(s)\left(  \Delta_{\varphi}-s(m-s)\right)
\Phi(s)f\right]  \right\vert _{x=0}.
\]
Since $\mathcal{P}(s)$ is holomorphic at $s=n/2+\ell/2$ (unless $s\in\Sigma$),
it follows that%
\[
\operatorname*{Res}_{s=m/2+\ell/2}\left[  S_{X}(s)f\right]
=\operatorname*{Res}_{s=m/2+\ell/2}\left[  \left.  x^{-s}\Phi(s)f\right\vert
_{x=0}\right]  .
\]
But%
\begin{align*}
\operatorname*{Res}_{s=m/2+\ell/2}\left(  \left.  x^{-s}\Phi(s)f\right\vert
\right)  _{x=0}  &  =\operatorname*{Res}_{s=m/2+\ell/2}\left(  \left.  \left[
x^{-s}\Phi(m-s)p_{\ell,s}f\right]  \right\vert _{x=0}\right) \\
&  =\operatorname*{Res}_{s=m/2+\ell/2}\left[  p_{\ell,s}f\right]
\end{align*}
and the claimed formula holds.
\end{proof}

To connect the scattering operator and the CR $Q$-curvature, we will also need
the following result about the pole of the scattering operator at $s=m$; this
result is a direct analogue of Proposition 3.7 in \cite{GZ:2003} but we give
the short proof for the reader's convenience.

\begin{proposition}
\label{prop.S1}Let $1$ denote the constant function on $M$. Then, the formula%
\[
S_{X}(m)1=-\lim_{s\rightarrow m}p_{m,s}(1)
\]
holds.
\end{proposition}

\begin{proof}
As $s\rightarrow m$ we have $\mathcal{P}(s)1\rightarrow1$. On the other hand,
for $s$ with $\left\vert s-m\right\vert <1/2$,
\[
\mathcal{P}(s)1=\sum_{k=0}^{m}x^{m-s+k}p_{k,s}(1)+x^{s}S_{X}(s)1+\mathcal{O}%
(x^{m+1/2}).
\]
This implies that
\[
\lim_{s\rightarrow m}\left[  x^{2m-s}p_{m,s}(1)+x^{m}S_{X}(s)1\right]  =0
\]
from which the claimed formula follows.
\end{proof}

\begin{remark}
\label{rem.FH}Note that, although $p_{m,s}$ has a pole at $s=m$, the limit
$\lim_{s\rightarrow m}p_{m,s}(1)$ exists. This implies that $P_{m,s}1$ (see
(\ref{eq.Pks})) has a first-order zero at $s=m$, i.e., $P_{m,s}1=(m-s)Q_{m,s}$
for a scalar function $Q_{m,s}$. The CR $Q$-curvature is then given by
$Q_{m,m}$ \cite{FH:2003}.
\end{remark}
\section{CR-Covariant Operators}

In this section we show that if $\varphi$ is an approximate solution of the
complex Monge-Amp\`{e}re equation in the sense discussed above, then the
residues of the scattering operator at $s=m/2+\ell/2$, $\ell=1,\cdots,m$ are
the CR-covariant differential operators $P_{k}$ defined in \cite{FH:2003}. In
order to do this we first recall Fefferman and Graham's \cite{FG:1985} set-up
for studying conformal invariants of compact manifolds and the construction of
the GJMS\ \cite{GJMS:1992} operators. We then recall its application to
CR-manifolds taking care that the arguments carry over from pseudoconvex
domains in $\mathbb{C}^{m}$ to the manifold setting studied here.

\subsection{The GJMS\ Construction}

\label{subsec.GJMS}

We begin by recalling Fefferman and Graham's construction of the ambient
metric and ambient space for a conformal manifold and the GJMS\ conformally
covariant operators on $\mathcal{C}$ obtained from this construction. Suppose
that $\left(  \mathcal{C},\left[  g\right]  \right)  $ is a conformal manifold
of signature $(p,q)$, i.e., a smooth manifold of dimension $N=p+q$ together
with a conformal class of pseudo-Riemannian metrics of signature $(p,q)$ on
$\mathcal{C}$. Fix a conformal representative $g_{0}$. The \emph{metric
bundle} $\mathcal{G}\subset S^{2}T^{\ast}\mathcal{C}$ is a bundle on
$\mathcal{C}$ with fibres%
\[
\mathcal{G}_{p}=\left\{  t^{2}g_{0}(p):t>0\right\}
\]
We denote by $\pi:\mathcal{G}\rightarrow M$ the natural projection. The
\emph{tautological metric} $G$ on $\mathcal{G}$ is given by%
\[
G(X,Y)=g(\pi_{\ast}X,\pi_{\ast}Y)
\]
for tangent vectors $X$ and $Y$ to $(p,g)\in\mathcal{G}$. There is a natural
$\mathbb{R}^{+}$-action $\delta_{s}$ on $\mathcal{G}$ given by $\delta
_{s}(p,g)=(p,s^{2}g)$.

The \emph{ambient space} over $\mathcal{C}$ is the space $\widetilde
{\mathcal{G}}=\mathcal{G}\times(-1,1)$. Note that the map $g\mapsto(g,0)$
imbeds $\mathcal{G}$ in $\widetilde{\mathcal{G}}$.

Fefferman and Graham proved the existence of a unique metric $\widetilde{g}$
of signature $(p+1,q+1)$ on $\widetilde{\mathcal{G}}$, the \emph{ambient
metric} on $\widetilde{\mathcal{G}}$ having the following three
properties:\newline(a) $i^{\ast}\widetilde{g}=G$\newline(b) $\delta_{s}^{\ast
}\widetilde{g}=s^{2}\widetilde{g}$\newline(c) $\operatorname*{Ric}%
(\widetilde{g})=0$ along $\mathcal{G}$ to infinite order if $N$ is odd, and up
to order $N/2$ if $N$ is even.

Here the uniqueness is meant in the sense of formal power series.

To define the GJMS\ operators, we first define spaces of homogeneous functions
on $\mathcal{G}$. For $w\in\mathbb{R}$ let $\mathcal{E}(w)$ denote the
functions $f$ on $\mathcal{G}$ homogeneous of degree $w$ with respect to
$\delta_{s}$ and smooth away from $0$. The GJMS\ operators $\mathcal{P}_{k}$
may be defined in two ways.

(1) Given $f\in\mathcal{E}(-N/2+k)$, extend $f$ to a function $\widetilde{f}$
homogeneous of the same degree on $\widetilde{\mathcal{G}}$, and set%
\begin{equation}
\mathcal{P}_{k}f=\left.  \widetilde{\Delta}^{k}\widetilde{f}\right\vert
_{\mathcal{G}} \label{eq.Pk}%
\end{equation}
where $\widetilde{\Delta}$ is the Laplacian for the ambient metric
$\widetilde{g}$ on $\widetilde{\mathcal{G}}$.

(2) Given $f\in\mathcal{E}(-N/2+k)$, $\mathcal{P}_{k}$ is the normalized
obstruction to extending $f$ to a smooth function $\widetilde{f}$ on
$\widetilde{\mathcal{G}}$ having the same homogeneity and satisfying
$\widetilde{\Delta}^{k}\widetilde{f}=0$.

The existence of GJMS\ operators was proven in \cite{GJMS:1992} for
$k=1,2,\cdots$ if $N$ is odd, and for $k=1,2,\cdots,N/2$ if $N$ is even.


\subsection{Application to CR-Manifolds}

Following \cite{GG:2005} we describe how the GJMS construction \cite{GJMS:1992} can be used to
prove the existence of CR-covariant differential operators. We begin with a
CR-manifold $M$ of dimension $2n+1$ and show how to construct a conformal
manifold $\mathcal{C}$ of dimension $2n+2$ and a conformal class of metrics
with signature $(2n+1,1)$ to which the GJMS\ construction may be applied. One
then \textquotedblleft pulls back\textquotedblright\ the GJMS\ operators to
$M$.

Recall that the \emph{canonical bundle }$K$ over $M$ is the bundle of
holomorphic $(n+1)$-forms generated by holomorphic forms of the type
$\theta\wedge\theta^{1}\wedge\cdots\wedge\theta^{n}$ where $\theta$ is a
contact form and $\left\{  \theta^{\alpha}\right\}  $ is a basis for
$\mathcal{H}$ of admissible $(1,0)$-forms.\ We denote by $K^{\ast}$ the
canonical bundle of $M$ with the zero section removed. The \emph{circle
bundle} $\mathcal{C}$ over $M$ is the bundle%
\[
\mathcal{C}=\left(  K^{\ast}\right)  ^{1/(n+2)}/\mathbb{R}^{+}.
\]
The circle bundle is an $S^{1}$-bundle over $M$, having real dimension $2m$ if
$m=n+1$. If we fix a contact form $\theta$ on $M$ (and hence a
pseudo-Hermitian structure on $M$), there is a corresponding section $\zeta$
of $K^{\ast}$ chosen so that $\theta$ is volume-normalized with respect to
$\zeta$. We denote by $\psi$ the angle determined by $\zeta(p)$ in each fibre
of $\mathcal{C}$ and define a fibre variable%
\[
\gamma=\frac{\psi}{n+2}%
\]
Note that $\gamma$ is canonically determined by $\theta$. Following Lee
\cite{Lee:1986}, let us define a canonical one-form $\sigma$ on $\mathcal{C}$
by%
\begin{equation}
(n+2)\sigma=(n+2)d\gamma+i\omega_{\alpha}^{\alpha}-\frac{1}{2(n+1)}%
R\theta\label{eq.sigma}%
\end{equation}
where $\omega_{\alpha}^{~\beta}$ is the connection one-form and $R$ is the
Webster scalar curvature of the pseudo-Hermitian structure $\theta$. The
mapping $\theta\mapsto g_{\theta}$ given by
\begin{equation}
g_{\theta}=h_{\alpha\beta}\theta^{\alpha}\cdot\theta^{\overline{\beta}%
}+2\theta\cdot\sigma\label{eq.gtheta}%
\end{equation}
(where $\cdot$ denotes the symmetric product) defines a mapping of
pseudo-Hermitian structures to Lorenz metrics which respects conformal
classes. One can now obtain GJMS\ operators on $\mathcal{C}$ using the
Fefferman-Graham construction.

\begin{remark}
\label{rem.r}It is immediate from formulas (\ref{eq.sigma}) and
(\ref{eq.gtheta}) that
\[
g_{\theta}(T,T)=-\frac{1}{(n+1)(n+2)}R.
\]
On the other hand, Farris \cite{Farris:1986} computed that, if $\theta$ is the
contact form induced by an approximate solution of the complex
Monge-Amp\`{e}re equation, then
\[
g_{\theta}(T,T)=2r
\]
where $r$ is the transverse curvature. It follows that the transverse
curvature is, in this case, an intrinsic pseudo-Hermitian invariant.
\end{remark}

To compute their pullbacks to $M$, we first note that the metric bundle
$\mathcal{G}$ of $(\mathcal{C},\left[  g\right]  )$ is diffeomorphic to
$\left(  K^{\ast}\right)  ^{1/(n+2)}$ and $\widetilde{\mathcal{G}}%
\simeq\left(  K^{\ast}\right)  ^{1/(n+1)}\times(-1,1)$. We define spaces of
functions%
\begin{align*}
\mathcal{E}(w,w^{\prime})  &  =\left\{  f\in\mathcal{C}^{\infty}((K^{\ast
})^{1/(n+2)}:f(\lambda\xi)=\lambda^{w}\overline{\lambda}^{w^{\prime}}%
f(\xi)\text{ for }\lambda\in\mathbb{C}^{\ast}\right\} \\
&  =\left\{  f\in\mathcal{E}(w+w^{\prime}):\left(  e^{i\phi}\right)  ^{\ast
}f(\xi)=e^{i\phi(w-w^{\prime})}f(\xi)\right\}
\end{align*}
We will primarily be concerned with functions in
\[
\mathcal{E}(w,w)=\left\{  f\in\mathcal{E}(w+w^{\prime}):\left(  e^{i\phi
}\right)  ^{\ast}f(\xi)=f(\xi)\right\}
\]
which descend to smooth functions on $M$.

For $k\in\mathbb{Z}$, we define
\begin{align*}
P_{w,w^{\prime}}  &  :\mathcal{E}(w,w^{\prime})\rightarrow\mathcal{E}%
(w-k,w^{\prime}-k)\\
f  &  \mapsto2^{-k}\mathcal{P}_{k}f,
\end{align*}
where $\mathcal{P}_{k}$ is defined in (\ref{eq.Pk}). Then choosing
$w=w^{\prime}=\left(  k-(n+1)\right)  /2$, we get operators $P_{k}$ defined on
$\mathcal{E}(-N/2+k)$.which are invariant under the circle action $\left(
e^{i\phi}\right)  ^{\ast}$ and hence may be viewed as smooth sections of a
density bundle over $M$. These operators $P_{k}$ are the CR-covariant
differential operators which we will connect to poles of the scattering operator.

If $X$ admits a globally defined approximate solution $\varphi$ of the
Monge-Amp\`{e}re equation, then for each $p\in M=\partial X$ there is a
neighborhood $U$ of $p$ and holomorphic coordinates $(z_{1},\cdots,z_{m})$
near $p$ so that $\varphi$ is an approximate solution of the Monge-Amp\`{e}re
equation in $U$. Let
\[
\theta=\left.  \frac{i}{2}\left(  \overline{\partial}-\partial\right)
\varphi\right\vert _{M}%
\]
be the induced pseudo-Hermitian structure on $M$, and let $\zeta=\left.
dz^{1}\wedge\cdots\wedge dz^{m}\right\vert _{M}$. Then $\theta$ is
volume-normalized with respect to $\zeta$.

Let us denote by $z_{0}$ the induced fibre coordinate of $\left(
\mathcal{K}^{\ast}\right)  ^{1/(n+2)}$ and let%
\[
Q=\left\vert z_{0}\right\vert ^{2}\varphi
\]
Then $Q$ is a globally defined smooth function on $\widetilde{\mathcal{G}}$
(which is diffeomorphic to $\mathbb{C}\times N$ for a collar neighborhood $N$
of $M$ in $X$) and the ambient metric on $\widetilde{\mathcal{G}}$ is the
K\"{a}hler metric associated to the K\"{a}hler form%
\[
\omega=i\partial\overline{\partial}Q
\]
where the corresponding metric $g_{\theta}$ on $\mathcal{C}$ is given by
(\ref{eq.gtheta}). The key computation linking the GJMS\ operators to the
Laplacian is given in Proposition 5.4 of \cite{GG:2005} and clearly
generalizes to our situation. Thus we have:

\begin{proposition}
\label{prop.link}If $u$ is a smooth function on $X$ then%
\[
\widetilde{\Delta}\left(  \left\vert z_{0}\right\vert ^{2w}\varphi
^{w}u\right)  =\left(  \left\vert z_{0}\right\vert ^{2w}\varphi^{w}\right)
\left(  \Delta_{\varphi}+w(n+1+w)\right)  u
\]
where $g$ is the metric associated to the K\"{a}hler form%
\[
\omega_{\varphi}=\frac{i}{2}\partial\overline{\partial}\log\left(
-1/\varphi\right)
\]

\end{proposition}

\section{Proofs of the Main Theorems}

Finally, we prove Theorems \ref{thm.1}, \ref{thm.2}, and \ref{thm.3}.

\begin{proof}
[Proof of Theorem \ref{thm.1}]The statement about the poles of $S_{X}(s)$ and
$s=m/2+k/2$ is proved in Proposition \ref{prop.S.poles}. If $g$ is a metric on
$X$ associated to the K\"{a}hler form $\omega=i\overline{\partial}\partial
\log(-1/\varphi)$ for a globally defined approximate solution of the
Monge-Amp\`{e}re equation, then the identification of the residues of
$S_{X}(s)$ with the CR-covariant differential operators of Fefferman and
Hirachi is a consequence of Proposition \ref{prop.link} and the second
characterization of the GJMS operators given in section \ref{subsec.GJMS}.
\end{proof}

\begin{proof}
[Proof of Theorem \ref{thm.2}]Owing to Proposition \ref{prop.S1}, \ it
suffices to identify $\lim_{s\rightarrow m}p_{m,s}1$ with the CR
$Q$-curvature. This is a consequence of Remark \ref{rem.FH}.
\end{proof}

\begin{proof}
[Proof of Theorem \ref{thm.3}]To prove Theorem \ref{thm.3}, let
\[
u_{s}=\mathcal{P}(s)1
\]
for $s$ real. Observe that $u_{s}$ is real and that $u_{s}\rightarrow1$ as
$s\rightarrow m$ uniformly on compact subsets of $X$. \ On the other hand, for
$s\neq m$ but $s$ close to $m$, $u_{s}$ takes the form
\begin{equation}
u_{s}\sim x^{m-s}F(s)+x^{s}G(s) \label{eq.us}%
\end{equation}
where $F(s)$ and $G(s)$ belong to $\mathcal{C}^{\infty}(X)$, $G(s)=S_{X}%
(s)1+\mathcal{O}(x)$ uniformly in $s$ near $m$, and
\begin{equation}
F(s)-1\sim\sum_{k\geq1}x^{k}F_{k}(s) \label{eq.usF}%
\end{equation}
where $F_{k}\in\mathcal{C}^{\infty}(M)$ and $F_{k}(s)\rightarrow0$ as
$s\rightarrow m$, save for the $F_{m}(s)$ term which obeys%
\begin{equation}
F_{m}(s)+S_{X}(s)1\rightarrow0\text{ as }s\rightarrow m \label{eq.usFm}%
\end{equation}
(see the proof of Proposition \ref{prop.S1}). Note that%
\[
\int_{M}F_{m}(s)\psi=\int p_{m,s}1~\psi.
\]
As in \cite{GZ:2003} we will prove Theorem \ref{thm.3} by computing
\[
\operatorname*{FP}_{\varepsilon\downarrow0}\left(  \int_{x>\varepsilon}\left[
\left\vert du_{s}\right\vert ^{2}-s(m-s)u_{s}u_{s}\right]  \omega^{m}\right)
\]
(where $\operatorname*{FP}$ denotes the Hadamard finite part) in two different
ways. First, we use Proposition \ref{prop.bp.extended} with $u_{1}=u_{2}%
=u_{s}$ to conclude that%
\begin{equation}
\operatorname*{FP}_{\varepsilon\downarrow0}\left(  \int_{x>\varepsilon}\left[
\left\vert du_{s}\right\vert ^{2}-s(m-s)u_{s}u_{s}\right]  \omega^{m}\right)
=-m\int_{M}S_{X}(s)1~\psi\label{eq.key.fp}%
\end{equation}
Secondly\ (and somewhat more painfully), we use the asymptotic expansion of
$u_{s}$ and the asymptotic form of the volume form $\omega^{m}$ directly to
conclude that%
\begin{equation}
\operatorname*{FP}_{\varepsilon\downarrow0}\left(  \int_{x>\varepsilon}\left[
\left\vert du_{s}\right\vert ^{2}-s(m-s)u_{s}u_{s}\right]  \omega^{m}\right)
=\frac{m}{2}L \label{eq.painful}%
\end{equation}
from which the desired equality will follow. The computations follow along the
lines of \cite{GZ:2003} with some trivial differences in the computation for
(\ref{eq.exp.2}) owing to the different form of the Laplacian on a complex
manifold. We give a summary in Appendix \ref{app.grunge}.
\end{proof}

\appendix{}

\section{CR $Q$-curvature and Asymptotic Volume}

\label{app.grunge}

The purpose of this appendix is to summarize the calculations leading to the
identity (\ref{eq.painful}) used in the proof of Theorem \ref{thm.3}. We will
show that%
\begin{equation}
\lim_{s\rightarrow m}s(m-s)\operatorname*{FP}_{\varepsilon\downarrow0}\left(
\int_{x>\varepsilon}u_{s}^{2}~\omega^{m}\right)  =-mc_{m}\int_{M}Q_{\theta
}^{CR}~\psi+mL/2 \label{eq.fp1}%
\end{equation}
and%
\begin{equation}
\lim_{s\rightarrow m}\operatorname*{FP}_{\varepsilon\downarrow0}\left(
\int_{\varepsilon}^{x_{0}}\left\vert du_{s}\right\vert ^{2}~\omega^{m}\right)
=-mc_{m}\int_{M}Q_{\theta}^{CR}~\psi\label{eq.fp2}%
\end{equation}
As in \cite{GZ:2003}, since $u_{s}\rightarrow1$ uniformly on compacts of $X$,
it suffices to compute the respective limits%
\begin{equation}
\lim_{s\rightarrow m}\left[  s(m-s)\operatorname*{FP}_{\varepsilon\downarrow
0}\left(  \int_{\varepsilon<x<x_{0}}u_{s}^{2}~\omega^{m}\right)  \right]
\label{eq.exp.1}%
\end{equation}
and%
\begin{equation}
\lim_{s\rightarrow m}\left[  \operatorname*{FP}_{\varepsilon\downarrow
0}\left(  \int_{\varepsilon<x<x_{0}}\left\vert du_{s}\right\vert ^{2}%
~\omega^{m}\right)  \right]  \label{eq.exp.2}%
\end{equation}
for any $x_{0}>0$. This reduction allows us to use boundary coordinates and
introduce asymptotic expansions for $u_{s}$ and $\omega^{m}$. In the
computations we make use of the simple formulas%
\begin{equation}
\operatorname*{FP}_{\varepsilon\downarrow0}\int_{\varepsilon}^{x_{o}%
}x^{m-2s+j}\frac{dx}{x}=\dfrac{x_{0}^{m-2s+j}}{m-2s+j} \label{eq.fpid.1}%
\end{equation}
(note that the finite part is independent of $x_{0}$ if $j=m$) and%
\begin{equation}
\operatorname*{FP}_{\varepsilon\downarrow0}\int_{\varepsilon}^{x_{0}}\frac
{dx}{x}=\log x_{0} \label{eq.fpid.2}%
\end{equation}
Setting $x=-\varphi$, we also have from (\ref{eq.volume}) that%
\[
\omega_{\varphi}^{m}=\frac{\eta}{x^{m}}\frac{dx}{x}\wedge(d\theta)^{n}%
\wedge\theta
\]
for $\eta\in\mathcal{C}^{\infty}(X)$ with $\psi=\left(  \left.  \eta
\right\vert _{M}\right)  \left(  d\theta\right)  ^{n}\wedge\theta$ (here
$\left.  \eta\right\vert _{M}=m/2^{n-1}$ in accordance with \ (\ref{eq.psi})).
We will write%
\[
\eta\sim\sum_{k\geq0}x^{k}\eta_{k}%
\]
for $\eta_{k}\in\mathcal{C}^{\infty}(M)$, so that
\[
L=\int_{M}\eta_{m}~\left(  d\theta\right)  ^{n}\wedge\theta.
\]

First, we consider (\ref{eq.exp.1}). In expanding the density
\[
u_{s}^{2}~\omega^{m}=f_{1}\frac{dx}{x}\wedge(d\theta)^{n}\wedge\theta
\]
asymptotically in $x$, we may neglect terms which are integrable, or terms
which give rise to finite parts which are holomorphic at $s=m$. It suffices
then to compute the coefficient of $x^{2m-2s}$ in the expansion for $f_{1}$,
since only the $x^{2m-2s}$ term will give rise to a finite part with pole at
$s=m$. Note that the resulting residue is independent of $x_{0}$ (see
(\ref{eq.fpid.1})). Since%
\begin{align*}
u_{s}^{2}  &  =x^{2m-2s}\left[  1+2(F(s)-1)+(F(s)-1)^{2}\right] \\
&  +2x^{2m}F(s)G(s)+x^{2s}G(s)^{2}%
\end{align*}
it suffices to examine the terms $x^{2m-2s}\eta_{m}$ and $2x^{2m-2s}F_{m}(s)$
in $f_{1}$. The first of these contributes $(m/2)L$ to (\ref{eq.exp.1}) and
the second contributes
\[
-\int_{M}S_{X}(m)1~\psi=-mc_{m}\int_{M}Q_{\theta}^{CR}~\psi.
\]
This leads to (\ref{eq.fp1})\ as claimed

Next, we consider (\ref{eq.exp.2}). From (\ref{eq.du.norm}), (\ref{eq.N.x}),
and the fact that
\[
\left\vert u_{m}\right\vert ^{2}=\frac{1}{4}\left\vert \left(  N-iT\right)
u\right\vert ^{2}%
\]
it follows that the density
\begin{equation}
\left\vert du_{s}\right\vert ^{2}\omega^{m}=\frac{1}{1+rx}\left\vert
x\partial_{x}u_{s}\right\vert ^{2}\omega^{m}+xH^{\alpha\overline{\beta}}%
(u_{s})_{\alpha}(u_{s})_{\overline{\beta}}\omega^{m} \label{eq.dense}%
\end{equation}
for a tensor $H^{\alpha\overline{\beta}}$ which is smooth in $x$ down to
$x=0$. We will show that the first right-hand term in (\ref{eq.dense}) leads
to the right-hand side of (\ref{eq.fp2})\ and the second right-hand term in
(\ref{eq.dense}) makes no contribution.

From the asymptotic form of $u_{s}$ (see (\ref{eq.us}), (\ref{eq.usF}),
\ref{eq.usFm})) we have%
\begin{equation}
x\partial_{x}u_{s}=x^{m-s}K_{1}(s)+x^{s}K_{2}(s) \label{eq.xds}%
\end{equation}
where%
\[
K_{1}(s)=(m-s)F(s)+xF_{x}(s)
\]
and $K_{1}(s)$ and $K_{2}(s)$ both approach zero as $s\rightarrow m$. For this
reason, writing%
\[
\left\vert x\partial_{x}u_{s}\right\vert ^{2}\omega^{m}=f_{2}\frac{dx}%
{x}\wedge(d\theta)^{n}\wedge\theta
\]
we need only consider the coefficient of $x^{2m-2s}$ in the expansion of
$f_{2}$ since all other terms give rise to terms whose finite parts vanish as
$s\rightarrow m$. On squaring (\ref{eq.xds}) we have%
\[
\left\vert x\partial_{x}u_{s}\right\vert ^{2}=x^{2m-2s}K_{1}(s)^{2}%
+2x^{m}K_{1}(s)K_{2}(s)+x^{2s}K_{2}(s)^{2}.
\]
We can drop terms containing $(m-s)^{2}$ times a holomorphic function since
these will vanish as $s\rightarrow m$, even if the power $x^{2s-2m}$ occurs in
$f_{2}$. Thus we need to compute the coefficient of $x^{m}$ in $K_{1}(s)^{2}$
to order $(m-s)$. This is $2(2m-s)(m-s)F_{m}(s)$ which contributes
$-(2m-s)\int_{M}F_{m}(s)~\psi$ to the finite part of $\int_{x>\varepsilon
}\left\vert x\partial_{x}u_{s}\right\vert ^{2}\omega^{m}$ and approaches
$-mc_{m}\int_{M}Q_{\theta}^{CR}~\psi$ as $s\rightarrow m$.

It remains to show that
\begin{equation}
\lim_{s\rightarrow m}\operatorname*{FP}_{\varepsilon\downarrow0}\left(
\int_{\varepsilon}^{x_{0}}\frac{rx}{1+rx}\left\vert x\partial_{x}%
u_{s}\right\vert ^{2}~\omega^{m}\right)  =0 \label{eq.fp.2b}%
\end{equation}
and%
\begin{equation}
\lim_{s\rightarrow m}\operatorname*{FP}_{\varepsilon\downarrow0}\left(
\int_{\varepsilon}^{x_{0}}xH^{\alpha\overline{\beta}}(u_{s})_{\alpha}%
(u_{s})_{\overline{\beta}}~\omega^{m}\right)  =0 \label{eq.fp.2c}%
\end{equation}
In the first case, we can use the analysis above to show that the coefficient
of $x^{2m-2s}$ in the expansion for $\left[  \left(  rx\right)
/(1+rx)\right]  \left\vert x\partial_{x}u_{s}\right\vert ^{2}~\omega^{m}$
vanishes as $(s-m)^{2}$ when $s\rightarrow m$, implying (\ref{eq.fp.2b}).
Introducing boundary local coordinates $y_{j}$, to prove that (\ref{eq.fp.2c})
holds it suffices to show that
\[
\lim_{s\rightarrow m}\operatorname*{FP}_{\varepsilon\downarrow0}\left(
\int_{\varepsilon}^{x_{0}}x\phi\frac{\partial u}{\partial y_{j}}%
~\frac{\partial u}{\partial y_{k}}\omega^{m}\right)  =0
\]
where $\phi$ is a smooth function supported in a local coordinate patch near
the boundary $M$. This follows from the fact that, in local coordinates
$(x,y)$ on $X$ in a neighborhood of $M$,
\[
\frac{\partial u}{\partial y_{j}}=x^{m-s}L_{1}(s)+x^{s}L_{2}(s)
\]
where both $L_{1}(s)$ and $L_{2}(s)$ vanish to order $(m-s)$ as $s\rightarrow
m$.

\end{document}